\documentclass[13pt,a4paper]{amsart}
\usepackage{fullpage}
\usepackage{mathrsfs}
\usepackage{CJK}
\usepackage{geometry}   
\usepackage{setspace}
\usepackage{amsfonts}
\usepackage{xy}
\usepackage{cite}
\usepackage{amssymb}
\usepackage{color, latexsym}
\usepackage{graphicx}
\usepackage{amsbsy,amscd}
\usepackage{fancyhdr}
\usepackage[colorlinks=true,
           citecolor=blue,
           linkcolor =red]{hyperref}
\xyoption{all}
\allowdisplaybreaks
\makeatletter
\@namedef{subjclassname@2020}{\textup{2020} Mathematics Subject Classification}
\makeatother

\xyoption{all}
\allowdisplaybreaks

\fontsize{9pt}{9pt}\selectfont

\def\bmu{\boldsymbol{\mu}}
\def\bgl{\boldsymbol{\lambda}}
\def\bx{\boldsymbol{x}}
\def\bn{\boldsymbol{n}}
\def\gl{\lambda}
\def\bz{\mathbb{Z}}
\def\bi{\boldsymbol{i}}
\def\ci{\mathcal{I}}
\def\fgl{\mathfrak{gl}}
\def\ft{\mathfrak{t}}

\def\rpn{\mathscr{P}_{r,n}}
\def\Number#1{\refstepcounter{equation}
              \leqno(\theequation)\if*#1%
              \else\def\@currentlabel{{\rm\theequation}}\label{#1}%
              \fi}

\newdimen\hoogte    \hoogte=13pt    
\newdimen\breedte   \breedte=13pt   
\newdimen\dikte     \dikte=0.5pt    
\newenvironment{point}[2]%
  {\ifx*#2\let\pointlabel\relax\else\def\pointlabel{#2}\fi
   \refstepcounter{equation}\trivlist
   \item[\hskip\labelsep\theequation.
         \ifx\pointlabel\relax\else\pointlabel\fi]
   \ignorespaces #1
  }{\relax}
\newenvironment{Young}{\begingroup
       \def\vr{\vrule height1.0\hoogte width\dikte depth 0.3\hoogte}
       \def\fbox##1{\vbox{\offinterlineskip
                    \hrule height1.0\dikte
                    \hbox to 1.0\breedte{\vr\hfill$##1$\hfill\vr}
                    \hrule height1.0\dikte}}
       \vtop\bgroup \offinterlineskip \tabskip=-\dikte \lineskip=-\dikte
            \halign\bgroup &\fbox{##\unskip}\unskip  \crcr}
     {\egroup\egroup\endgroup}

\def\diagram(#1){\vcenter{\begin{Young}#1\cr\end{Young}}}

\def\twodiagram(#1|#2){\,\vcenter{\begin{Young}#1\cr\end{Young}}\,;\,
\vcenter{\begin{Young}#2\cr\end{Young}}}

\def\tridiagram(#1|#2|#3)
{\,\vcenter{\begin{Young}#1\cr\end{Young}}\,;\,
\vcenter{\begin{Young}#2\cr\end{Young}}\,;\,
\vcenter{\begin{Young}#3\cr\end{Young}}}
\newenvironment{cyoung}{\begingroup
       \def\vr{\vrule height1.0\hoogte width\dikte depth 0.8\hoogte}
       \def\fbox##1{\vbox{\offinterlineskip
                    \hrule height1.0\dikte
                    \hbox to 1.5\breedte{\vr\hfill$##1$\hfill\vr}
                    \hrule height0.80\dikte}}
       \vbox\bgroup \offinterlineskip \tabskip=-\dikte \lineskip=-\dikte
            \halign\bgroup &\fbox{##\unskip}\unskip  \crcr }
       {\egroup\egroup\endgroup}

\def\cdiagram#1{\relax\ifmmode\vcenter{\,\begin{cyoung}#1\end{cyoung}\,}\else%
              $\vcenter{\,\begin{cyoung}#1\end{cyoung}\,}$\fi}

\def\twocdiagram(#1,#2){\,\vcenter{\begin{cyoung}#1\cr\end{cyoung}}\,;\,
\vcenter{\begin{cyoung}#2\cr\end{cyoung}}\quad}
\makeatother

\numberwithin{equation}{section}
\swapnumbers
\newtheorem{theorem}[equation]{Theorem}
\newtheorem{lemma}[equation]{Lemma}
\newtheorem{proposition}[equation]{Proposition}
\newtheorem{corollary}[equation]{Corollary}
\theoremstyle{definition}
\newtheorem{definition}[equation]{Definition}
\newtheorem{example}[equation]{Example}
\theoremstyle{remark}
\newtheorem{remark}[equation]{Remark}

\begin{document}
\setlength{\itemsep}{-0.25cm}
\fontsize{12}{\baselineskip}\selectfont
\setlength{\parskip}{0.30\baselineskip}
\vspace*{-4mm}
\title[Foulkes characters]{\fontsize{10}{\baselineskip}\selectfont Revisiting Foulkes
characters of wreath products}
\author{Deke Zhao}
\address{\bigskip\hfil\begin{tabular}{l@{}}
          School of Applied Mathematics\\
            Beijing Normal University at Zhuhai, Zhuhai, 519087\\
             China\\
             E-mail: \it deke@amss.ac.cn \hfill
          \end{tabular}}
\thanks{Supported by Guangdong Basic and Applied Basic Research
 (2023A1515010251) and the National Natural Science Foundation of China (Grant No.  11871107)}

\begin{abstract}The article is concerned with the Foulkes characters of wreath products, which are block characters of wreath products, i.e., the positive-definite class functions depending only on the length of its elements. Inspired by the works of Gnedin--Gorin--Kerov and Miller, we introduce two specializations of the Schur--Weyl--Sergeev duality for wreath products and obtain two families of block characters, which provide a decomposition and an alternative construction of the Foulkes characters of wreath products. In particular, we give alternative proofs on some remarkable properties of the Foulkes characters. Along the way, we show that the Foulkes characters are the extreme rays of the cone of the block characters of wreath products and construct the representations with traces being the Foulkes characters via the coinvariant algebra of wreath products.
\end{abstract}
\maketitle
\vspace*{-10mm}
\section{Introduction}\label{Sec:Introduction}
Let $n$ be a positive integer and $\mathfrak{S}_n$ be the symmetric group of degree $n$.
The Foulkes characters for $S_n$ were discovered by Foulkes \cite{Foulkes} as part of the
study of the descent patterns in the permutation groups. These characters have been shown
 enjoy many remarkable properties \cite{Kerber-T} and have been used by Gessel--Reutenauer
  \cite{Gessel-R} to enumerate permutations with given descents and conjugacy classes and
  by Stanley \cite{Stanley} to enumerate alternating permutations by cycle types.
In \cite{GGK}, Gnedin--Gorin--Kerov provided an alternative description of the Foulkes
characters via the Schur--Weyl duality, which also provides a new realization of representations
 of $S_n$ with traces being Foulkes characters via the coinvariant algebra of $\mathfrak{S}_n$.

In \cite{M15}, Miller introduced the Foulkes characters for complex reflection groups and
showed that they enjoy many remarkable properties in common with Foulkes characters for symmetric groups. Motivated by the works of Gnedin--Gorin--Kerov and Miller, the paper aims to investigate the Foulkes characters for wreath products via the Schur--Weyl--Sergeev duality between the wreath  products and universal enveloping algebras of certain Levi Lie algebras. This method yields   an alternative proof of some properties of the Foulkes characters for wreath products and   provide a realization of representations of wreath products with traces being the Foulkes
  characters via the coinvariant algebra of wreath products. Furthermore, we hope this method
   may help us to understand the products of Foulkes characters and the inner product introduced     by Miller in \cite{M21}.

We shall now explain in more details our results. Let $r$ be a positive integer and $W_{r,n}$
 be the complex reflection group of type $G(r,1,n)$ in Shephard--Todd's classification \cite{ST},
  that is, $W_{r,n}=\mathbb{Z}_r\rtimes \mathfrak{S}_n$ is the wreath product of  $\mathfrak{S}_n$ with the cyclic group $\mathbb{Z}_r$ of order $r$.

 By a character of $W_{r,n}$, we mean a function $\chi: W_{r,n}\rightarrow \mathbb{C}$
  satisfying the following conditions:
 \begin{enumerate}
   \item[(C1)]$\chi$ is constant on conjugacy classes: $\chi(\sigma\tau\sigma^{-1})=\chi(\tau)$
    for all $\sigma,\tau\in W_{r,n}$;
   \item[(C2)] $\chi$ is non-negative definite: $\left(\chi(\sigma_i\sigma_j^{-1})\right)_{i,j=1}^k$
    is hermitian and non-negative definite for all $\sigma_1, \ldots, \sigma_k\in W_{r,n}$, $k\in \mathbb{Z}_{>0}$.
 \end{enumerate}
By a \emph{normalized character} $\chi$ of $W_{r,n}$ we mean that $\chi$ is a character
with  $\chi(\boldsymbol{e})=1$, where $\boldsymbol{e}$ is the unity of $W_{r,n}$.
Furthermore, an  \emph{extremal character} is a character $\chi$ if
it cannot be written as a combination $a\chi_1+ (1-a)\chi_2$ with
characters $\chi_1\neq \chi_2$ and $a\in (0,1)$.

It is well-known that conjugate classes of $W_{r,n}$ are parameterized by the set
 $\rpn$ of all $r$-multipartitions of $n$. Following Miller \cite{M17}, we define the
 length function $\ell$ on $W_{r,n}$ as follows:
 \begin{equation}\label{Equ:Def-Length}
   \ell_n(w):=\text{number of parts of $\lambda^{(0)}$},
 \end{equation}
when $w$ is an element of $W_{r,n}$ bing of type $\bgl=(\lambda^{(0)};\ldots;\lambda^{(r-1)})$.
A character of $W_{r,n}$ is said to be a \emph{block character} provided it depends  on an element
 $w$ of $W_{r,n}$ only through $\ell_n(w)$. Thanks to Miller's work \cite{M15}, the Foulkes
  characters $\phi_0^n, \phi_1^n,\ldots,\phi_n^n$ of $W_{r,n}$ are block characters.

For a non-negative integer $k$, let $V^{(0)}=\mathbb{C}^{k+1}$, $V^{(1)}=\cdots=V^{(r-1)}
=\mathbb{C}^k$ be superspaces concentrated either all in degree $\bar{0}$ or all in
degree $\bar{1}$. Put
\begin{eqnarray*}
&&V=V^{(0)}\oplus V^{(1)}\oplus\cdots\oplus V^{(r-1)},\\
&&\mathfrak{g}=\mathfrak{gl}(k+1)\oplus\mathfrak{gl}(k)\oplus\cdots\oplus \mathfrak{gl}(k)\hookrightarrow \mathfrak{gl}(rk+1).
\end{eqnarray*}
Let $\mathcal{U}(\mathfrak{g})$ be its universal enveloping algebra. Then $W_{r,n}$ and $\mathcal{U}(\mathfrak{g})$ act naturally on $V^{\otimes n}$ and the Schur--Weyl--Sergeev duality for $W_{r,n}$ holds (see Propositions~\ref{Prop:Schur-Weyl} and \ref{Prop:Schur-Sergeev}). It is interesting that these dualities give us two families of block characters $\chi_k^n, \widetilde{\chi}_k^n$ ($k\geq 0$) for $W_{r,n}$ (see Propositions~\ref{Prop:Chi} and \ref{Prop:Chi-signed}). Then $\phi_0^n, \phi_1^n, \ldots, \phi_n^n$ of $W_{r,n}$ can be redefined in terms of $\chi^n_0, \chi_1^n, \ldots, \chi^n_n$ (see Definition~\ref{Def:Foulkes}).

 Combining the Schur--Weyl duality and Definition~\ref{Def:Foulkes}, we can describe explicitly the decomposition of the Foulkes characters into direct summands of the irreducible characters of $W_{r,n}$ (see Proposition~\ref{Prop:Foulkes-decom}). Then, by using the block characters $\widetilde{\chi}^n_0, \widetilde{\chi}_1^n, \ldots, \widetilde{\chi}^n_n$, we can define
 another family block functions $\widetilde{\chi}^n_0, \widetilde{\chi}_1^n, \ldots, \widetilde{\chi}^n_n$ on $W_{r,n}$ (see Definition~\ref{Def:Signed-Foulkes}),  which are shown to be block characters of $W_{r,n}$ and $\widetilde{\chi}^n_k=\chi^n_{n-k}$ for $k=0, 1, \ldots,n$ (see Proposition~\ref{Them:Foulkes=sign-Foulkes}).

This approach not only gives a new realization of the $W_{r,n}$-representations with trace being Foulkes characters, but also provides alternative proofs on some remarkable properties of the Foulkes characters (see Theorem~\ref{Them:Properties}), which are first proved by Miller in \cite{M15}. Then we show that the convex set of normalized block characters of $W_{r,n}$ is a simplex whose extreme points (characters) are the normalized Foulkes characters. Finally let $\mathbb{C}[\boldsymbol{x}]^{\mathrm{coinv}}$ be the coinvariant algebra associated to $W_{r,n}$. Following Gnedin--Gorin--Kerov \cite{GGK}, we define a filtration
\begin{equation*}
 \{0\}=\mathcal{F}_n(-1)\subset \mathcal{F}_n(0)\subset \cdots \mathcal{F}_n(n)=\mathbb{C}[\boldsymbol{x}]^{\mathrm{coinv}}
\end{equation*}
and show that trace of the $W_{r,n}$-representation $\mathcal{F}_n(k)/\mathcal{F}_n(k-1)$ is the Foulkes character $\phi_k^n$ for $k=0,1, \ldots, n$ (see Theorem~\ref{Them:Filtration=Foulkes}).

Let us remark that Gnedin--Gorin--Kerov \cite{GGK} have investigated the limit shapes of the Foulkes characters of $\mathfrak{S}_n$ and  provided a classification of block characters of the infinite symmetric group. It is natural and interesting to study the limit shapes of the Foulkes characters of $W_{r,n}$ and the block characters of the infinite wreath product.

Throughout this paper $\xi$ is a primitive $r$th root of unity,  $\mathbb{N}=\{1, 2, \ldots\}$ and $\mathbb{C}$ is the field of complex numbers. For any $n\in\mathbb{N}$, we let $\bn=\{1, 2, \ldots,n\}$ and denote by $\mathfrak{S}_n$ the symmetric group on $\bn$. We always represent the elements of the $r$th cyclic group $\mathbb{Z}_r$ by the elements of $\{0,1,\ldots, r-1\}$ and write an element $\sigma$ of the symmetric group $\mathfrak{S}_n$ as the one-line form $\sigma=\sigma(1)\cdots\sigma(n)$.

The remainder of the paper is organized as follows. We begin with a brief review on the wreath products and fixed our notations in Section~\ref{Sec:Wreath-product}. In Section~\ref{Sec:Block-Schur-Sergeev}, we introduce two families of block characters $\chi_k^n$ and $\widetilde{\chi}_k^n$ ($0\leq k\leq n$) of $W_{r,n}$ and investigate their branching rules. The Foulkes characters of $W_{r,n}$ is redefined and investigated via the  block characters $\chi_k^n$ ($0\leq k\leq n$) in Section~\ref{Sec:Foulkes-characters}. Section~\ref{Sec:Duality-Foulkes} deals with a duality between Foulkes characters by applying the block characters $\widetilde{\chi}_k^n$ ($0\leq k\leq n$) and the Schur--Sergeev duality. We provide the alternative proofs of some remarkable properties of the Foulkes characters of $W_{r,n}$ and construct $W_{r,n}$-representations with traces being the Foulkes characters respectively in Sections~\ref{Sec:Foulkes-Properties} and \ref{Sec:Coinvariant} respectively.


\section{The wreath products}\label{Sec:Wreath-product}
In this section we recall some combinatoric notations and the definition of descent of elements of wreath product.

\begin{point}{}*Recall that a partition $\gl=(\gl_1, \gl_2, \ldots)$ of $n$, denote $\gl\vdash n$, is a weakly decreasing sequence of  nonnegative integers such that $|\gl|=\sum_{i\geq1}\gl_i=n$ and we write $\ell(\gl)$ the length of $\gl$, i.e. the number of nonzero parts of $\gl$. The \emph{conjugate} of  $\lambda$ is the partition $\overline{\lambda}=(\overline{\lambda}_1,\overline{\lambda}_2,\ldots)$ with
\begin{equation*}
  \overline{\lambda}_i=|\{j\mid \lambda_j\geq i\}|.
\end{equation*}
By an $r$-multipartition $\bgl=(\lambda^{(0)}; \ldots;\lambda^{(r-1)})$ of $n$, we mean an $r$-tuple of partitions such that $|\bgl|=\sum_{i=0}^{r-1}|\lambda^{(i)}|=n$.   We call $\gl^{(i)}$ the \emph{$i$th component} of $\bgl$ and denote by $\rpn$  the set of all  $r$-multipartitions of $n$.

The \textit{Young diagram} of
an $r$-multipartition $\bgl$ is the ordered $r$-tuple of the Young diagrams of its components, or equivalently, the set
\begin{equation*}
  \bgl:=\{(i,j,c)\in\bz_{>0}\times\bz_{>0}\times \{1, \ldots, r\}|1\le j\le\lambda^{(c)}_i\}.
\end{equation*}
We may and will identify $\bgl$ with its Young diagram. The elements of $\bgl$ are the boxes of $\bgl$; more generally, a box is any element of
$\mathbb{Z}_{>0}\times \mathbb{Z}_{>0}\times \{1,\ldots,r\}$. Recall that a box $\Box\notin\bgl$ is \emph{addable} for $\bgl$ if
$\bgl\cup\{\Box\}$ is the diagram of an $r$-multipartition, and denote by $\mathscr{A}(\bgl)$ the set of all addable boxes for $\bgl$; similarly,
$\Box\in\bgl$ is \emph{removable} for $\bgl$ if
$\bgl-\{\Box\}$ is the diagram of an $r$-multipartition,  and denote by
$\mathscr{R}(\bgl)$ the set of all removable boxes for $\bgl$.
\end{point}
\begin{example} Let $\bgl=((3,2,2);(2,1))$ be a $2$-multipartition of $8$. Then its Young diagram
\begin{eqnarray*}
&&\protect{\bgl=\left(\twodiagram(&&\cr &\cr&|&\cr)\right)}}, \protect{\mathscr{R}(\bgl)=\tiny\tiny\tiny{\left(\twodiagram(&&-\cr&\cr &-|&-\cr -)
\right)}},\protect{\mathscr{A}(\bgl)=\tiny{\left(\,\twodiagram(&&&+\cr &&+\cr&\cr +|&&+\cr &+\cr +)\right)},
\end{eqnarray*}
where the $\boxminus$ (resp. $\boxplus$) means the box is removable (resp., addable).
\end{example}

\begin{point}{}*The wreath product $W_{r,n}:=\mathbb{Z}_r\wr\mathfrak{S}_n$ is the semidirect product of $\mathbb{Z}^n_r$, the $n$th direct power of $\mathbb{Z}_r$, with $\mathfrak{S}_n$, obtained by the natural $\mathfrak{S}_n$-action on the $n$ copies of $\mathbb{Z}_r$, namely $W_{r,n}$ is the set of products $\sigma\times\boldsymbol{z}$ of a permutation $\sigma=\sigma(1)\sigma(2)\cdots\sigma(n)\in \mathfrak{S}_n$ and an $r$-tuple $\boldsymbol{x}=(x_1,x_2, \ldots,x_n)$ of elements $x_i\in\mathbb{Z}_r$. As a convention, we set $\sigma(n+1)=n+1$ and  $x_{n+1}=0$. With the notation $\sigma\times\boldsymbol{z}$, the product is defined by
\begin{equation*}
(\sigma\times\boldsymbol{x})\cdot(\tau\times\boldsymbol{y}):=\sigma\tau\times (\boldsymbol{x}+\sigma(\boldsymbol{y})),
\end{equation*}
where $\sigma\tau$ is evaluated from right to left, $\sigma(\boldsymbol{y})=(y_{\sigma(1)},y_{\sigma(2)},\ldots,y_{\sigma(n)})$ and the addition is the coordinate-wise modulo $r$. Note that $W_{r,n}$ is the complex reflection group of type $G(r,1,n)$ in Shephard--Todd's classification in \cite{ST}), which is generated by $s_0$, $s_1$, $\ldots$, $s_{n-1}$ with relations:
  \begin{eqnarray*}
 &&s_0^r=\boldsymbol{e}, \quad s_i^2=\boldsymbol{e} \text{ for } 1\leq i<n,\\
 &&s_is_j=s_js_i \text{ for } |i-j|>1,\\
 && s_0s_1s_0s_1=s_1s_0s_1s_0,\\
 && s_is_{i+1}s_i=s_{i+1}s_is_{i+1} \text{ for } 1\leq i<n-1.
 \end{eqnarray*}

It should be pointed out that the elements of $W_{r,n}$ can be take as those $n\times n$ matrices with exactly one nonzero entry in each row and each column, and such that each of these non-zero entries is an $r$th roots of unity. With this definition, the product in $W_{r,n}$ is simply matrix multiplication. $W_{r,n}$ can also be viewed as a group of \emph{$r$-colored permutations}, consisting of all the permutations of the set of $rn$ colored digits
\begin{equation*}
 \{(i, z) : 1 \leq i \leq n, z\in \mathbb{Z}_r\},
\end{equation*}
which are $\mathbb{Z}_r$-equivariant, in the sense
that if $\pi(i, z) = (j, y)$ then $\pi(i, z + x) = (j, y + x)$ for all $x\in \mathbb{Z}_r$.
This explains the fact that this group is also called \emph{the group of $r$-colored
permutations}. Sometimes we write its elements as the one-line form
\begin{equation*}
  w=w(1)^{c_1}w(2)^{c_2}\cdots w(n)^{c_n}.
\end{equation*}
For example
\begin{equation*}
 w=(2\,1\,4\,5\,3\,7\,6\,9\,10\,8)\times(0,0,1,1,0,1,1,0,1,1)
 =2^0\,1^0\,4^1\,5^1\,3^0\,7^1\,6^1\,9^0\,10^1\,8^1\,11^0.
\end{equation*}

For an element $w=\sigma(1)^{c_1}\sigma(2)^{c_2}\cdots\sigma(n)^{c_n}$ of $W_{r,n}$, first write $\sigma=\sigma(1)\sigma(2)\cdots\sigma(n)$ a product of disjoint cycles, and then
provide the entries with their original colors, thus obtaining colored cycles. The color of a cycle is simply the sum of all the colors of its entries. For every $i=0, \ldots, r-1$, let $\mu^{(i)}$ be the partition formed by the lengths of the cycles of $g$ having
color $i$. We may thus associate $g$ with the $r$-multipartition $\bmu=(\mu^{(0)};
\mu^{(1)}; \ldots;\mu^{(r-1)})$, which is called \textit{the type of $g$}.
It is known that the conjugacy classes
of $W_{r,n}$, as well as its irreducible characters, are indexed by the $r$-multipartitions of $n$ (see e.g. \cite[Appendix B]{Macdonald}). We shall denote by $C_{\bmu}$ the conjugacy class of $W_{r,n}$ indexed by $\bmu$.
\end{point}

\begin{example}
 The decompositions into colored cycles of the following elements in $W_{3,12}$
\begin{eqnarray*}
&&w_1=2^1\,1^0\,4^0\,5^1\,3^1\,7^0\,6^0\,9^0\,10^0\,8^0\,12^0\,11^1
=(1^0\,2^1)(3^1\,4^0\,5^1)(6^0\,7^0)(8^0\,9^0\,10^0)(11^1\,12^0),\\
&&w_2=2^0\,1^0\,4^1\,5^1\,3^1\,7^1\,6^0\,9^1\,10^1\,8^0\,12^1\,11^0
=(1^02^0)(3^14^15^1)(6^07^1)(8^09^110^1)(11^012^1),\\
&&w_3=2^1\,3^1\,1^1\,5^0\,4^0\,7^1\,6^0\,9^1\,8^0\,11^0\,12^1\,10^1
=(1^1\,2^1\,3^1)(4^0\,5^0)(6^0\,7^1)(8^0\,9^1)(10^1\,11^0\,12^1).
\end{eqnarray*}
All of these elements are of type $\bmu=(\mu^{(0)};\mu^{(1)};\mu^{(2)})$, where $\mu^{(0)}=(3,2)$, $\mu^{(1)}=(2,2)$, and $\mu^{(2)}=(3)$. Thus $w_1, w_2, w_3$ in the  conjugacy class $C_{\bmu}$.
\end{example}

\begin{definition}[\protect{\cite[Definition~2]{St}}]\label{Def:Descent}
An integer $i\in\bn$ is a \textit{descent} of
 \begin{equation*}
 w=w(1)^{c_1}w(2)^{c_2}\cdots w(n)^{c_n}\in W_{r,n}
\end{equation*}
if $c_{i}>c_{i+1}$ or $c_{i}=c_{i+1}$ and $w(i)>w(i+1)$, where $w(n+1):=n+1$ and $c_{n+1}=0$. Denote by $\mathrm{Des}(w)$ the set of descents of $w$ and by $\mathrm{des}(w)$ the number of descents (also called the \emph{descent number}) of $w$. The \emph{Eulerian number} $E_{r,n}(k)$ counts the number of elements in $W_{r,n}$ having $k$ descents:
 \begin{equation*}
   E_{r,n}(k)=\left|\{w\in W_{r,n}|\mathrm{des}(w)=k\}\right|.
 \end{equation*}
\end{definition}
Recall that we let $w(n+1)=n+1$ and  $a_{n+1}=0$.
Define an ordering $\prec$ on the alphabet $\{i^j|i\in\boldsymbol{n},j\in \mathbb{Z}_r\}$ by setting
\begin{equation*}
1^0\prec2^0\prec\cdots\prec n^0\prec n+1^0\prec 1^1\prec 2^1
 \prec\cdots\prec n^1\prec\cdots\prec 1^{r-1}\prec 2^{r-1}\prec\cdots\prec n^{r-1}.
\end{equation*}
Then there is an alternative definition of descents:
\begin{equation*}
  \mathrm{Des}(w)=\{i\in\boldsymbol{n}|w(i+1)^{c_{i+1}}\prec w(i)^{c_i}\}
\end{equation*}for $w=w(1)^{c_1}w(2)^{c_2}\cdots w(n)^{c_n}\in W_{r,n}$.
 For example, let $w=3^0\,2^0\,1^0\,4^2\,6^2\,5^1\in W_{3,6}$. Then $\mathrm{Des}(w)=\{1,2,5,6\}$, so that $\mathrm{des}(w)=4$.

\begin{remark}\label{Remark:Descent}Definition~\ref{Def:Descent} is slightly differential from the one used by Bagno--Biagioli in \cite{BB} to construct the color descent representations of complex reflection groups. Note that if $i$ is a descent of  $w=w(1)^{c_1}w(2)^{c_2}\cdots w(n)^{c_n}\in W_{r,n}$, then either $(r-1)-c_{i}<(r-1)-c_{i+1}$ or $c_{i}=c_{i+1}$ and $w(i)>w(i+1)$. Thus the two definitions are exactly equivalent.
\end{remark}

 Thanks to \cite[Lemma~16]{St},
 we have the following recurrence for the Eulerian number $E_{r,n}(k)$:
\begin{equation}\label{Equ:Euler-recurrence}
 E_{r,n}(k)=(rk+1)E_{r,n-1}(k)+(r(n+1)-(rk+1))E_{r,n-1}(k-1).
\end{equation}

\section{Two families of block characters}\label{Sec:Block-Schur-Sergeev}
  The section aims to construct two families of block characters via the Schur--Weyl--Sergeev reciprocity between $W_{r,n}$ and the universal enveloping algebra of certain Levi subalgebra of the Lie algebra $\mathfrak{gl}_{rk+1}(\mathbb{C})$ for all non-negative integer $k$. The first family of block characters enable us to give a new construction of the representations of $W_{r,k}$ such that their characters are Foulkes characters of $W_{r,n}$, while the others can be used to define the signed Foulkes characters of $W_{r,n}$, which will be used to describe a duality between the Foulkes characters.

\begin{point}{}* For a non-negative integer $k$,
let $\bx_k=\bx^{(0)}\cup\cdots\cup \bx^{(r-1)}$ be a set of $rk+1$ variables, where
\begin{eqnarray*}
  &&\bx^{(0)}=\left\{x^{(0)}_1, \cdots, x_{k+1}^{(0)}\right\}\quad\text{ and }\quad \bx^{(i)}=\left\{x^{(i)}_1, \cdots, x_{k}^{(i)}\right\} \text{ for }1\leq i\leq r-1.
\end{eqnarray*}
We say that the variables in $\bx^{(i)}$ are of \emph{color} $i$ and linearly order these variables by the rule \begin{eqnarray*}
  x_{a}^{(i)}<x_{b}^{(j)}&\text{ if and only if }& i<j\text{ or }i=j\text{ and }a<b.
\end{eqnarray*}

Let $\bgl=(\lambda^{(1)};\ldots;\lambda^{(r)})$ be an $r$-multipartition.  An \emph{$\boldsymbol{x}_k$-tableau} $\ft=(\ft^{(0)};\ldots;\ft^{(r-1)})$ of shape $\bgl$ is obtained by filling each box of the Young diagram of $\bgl$ with one variable from $\bx_k$, allowing repetitions. We say that an  $\boldsymbol{x}_k$-tableau $\ft=(\ft^{(0)};\ldots;\ft^{(r-1)})$ of shape $\bgl$ is \emph{row-semistandard} (resp. \emph{column-semistandard}) if, for each $i=0,\ldots, r-1$,
\begin{enumerate}
  \item[(1)] $\ft^{(i)}$ only contains variables from $\bx^{(i)}$;
  \item[(2)] The entries of $\ft^{(i)}$ are nondecreasing along rows (resp. down columns), strictly increasing in down columns (resp. along rows).
\end{enumerate}
Denote by  $\mathrm{Rstd}_k(\bgl)$ (resp. $\mathrm{Cstd}(\bgl)$)  the set  of row-semistandard (resp. column-semistandard) $\boldsymbol{x}_k$-tableaux of shape $\bgl$ and by  $s_k(\bgl)$ (resp. $c_k(\bgl)$) its cardinality. Furthermore, let
\begin{eqnarray*}
&&\boldsymbol{Y}^k_{r,n}=\{(\lambda^{(0)};\ldots;\lambda^{(r-1)})\vdash n\,\mid\, \ell(\lambda^{(0)})\leq k+1, \ell(\lambda^{(i)})\leq k \text{ for }i=1,\ldots,r-1\};\\
&&\widetilde{\boldsymbol{Y}}^k_{r,n}=\{(\lambda^{(0)};\ldots;\lambda^{(r-1)})\vdash n\,\mid\, \ell(\overline{\lambda^{(0)}})\leq k+1, \ell(\overline{\lambda^{(i)}})\leq k \text{ for }i=1,\ldots,r-1\}.
\end{eqnarray*}
Then $s_{k}(\bgl)\neq0$ (resp. $c_{k}(\bgl)\neq0$) if and only if $\bgl\in \boldsymbol{Y}^k_{r,n}$ (resp. $\bgl\in \widetilde{\boldsymbol{Y}}^k_{r,n}$).
\end{point}

\begin{example}Let $\bgl=((3,2,2);(2,1))$ and $k=2$. Then
\begin{equation*}\left(\twocdiagram(
x_{1}^{(0)}&x_{1}^{(0)}&x_{1}^{(0)}\cr
 x_{2}^{(0)}&x_2^{(0)}\cr x_{3}^{(0)}&x_{3}^{(0)},
 x_{1}^{(1)}&x_{1}^{(1)}\cr x_{2}^{(1)})
   \right) \quad\text{ and }\quad
   \left(\twocdiagram(x_1^{(0)}&x_2^{(0)}&x_3^{(0)}\cr
 x_{1}^{(0)}&x_2^{(0)}\cr x_{1}^{(0)}&x_3^{(0)},x_1^{(1)}&x_2^{(1)}
 \cr x_2^{(1)})
   \right)
\end{equation*}
are row-semistandard and column-semistandard $\boldsymbol{x}_k$-tableaux of shape $\bgl$ respectively.
\end{example}

\begin{point}{}* Now  let $V^{(0)}$ be a $k+1$-dimensional $\mathbb{C}$-space  with a basis
\begin{equation*}
  \mathfrak{B}^{(0)}=\left\{v^{(0)}_1, \ldots, v_{k+1}^{(0)}\right\}
\end{equation*}
 and let $V^{(i)}$ be a $k$-dimensional $\mathbb{C}$-space  with a basis
\begin{equation*}
  \mathfrak{B}^{(i)}=\left\{v^{(i)}_1, \ldots, v_{k}^{(i)}\right\}
\end{equation*}for $i=1, \ldots, r-1$. We define a total order on the set  $\coprod_{i}\mathfrak{B}^{(i)}$ by the lexicographic order on the set of pairs $(i,j)$ such that $0\leq i\leq r-1, 1\leq j\leq k+1$, and we identify the set $\coprod_{i=0}^{r-1}\mathfrak{B}^{(i)}$ with the basis $\mathfrak{B}=\{v_1, \ldots, v_{rk+1}\}$ of an $rk+1$-dimensional $\mathbb{C}$-space $V$ by the total order. Hence we have $V=\oplus_{i=0}^{r-1}V^{(i)}$. Further, we define a function $c: \mathfrak{B}\rightarrow \mathbb{N}$ by  $c(v_{j}^{(i)})=i$.

Let $\mathcal{I}(k;n)=\{\bi=(i_1, \ldots, i_n)|1\leq i_t\leq rk+1, 1\leq t\leq n\}$.
For $\bi=(i_1, \ldots, i_n)\in \mathcal{I}(k;n)$, we write $v_{\bi}=v_{i_1}\otimes \cdots\otimes v_{i_n}$ and put $c_a(v_{\bi})=c(v_{i_a})$. Then $\mathfrak{B}^{\otimes n}=\{v_{\bi}|\bi\in \ci(k;n)\}$ is a basis of $V^{\otimes n}$. We may and will identify $\mathfrak{B}^{\otimes n}$ with $\ci(k;n)$, that is, we will write $v_{\bi}$ by $\bi$, $c_a(v_{\bi})$ by $c_a(\bi)$, etc., if there are no confusions.

Thanks to \cite[\S2.7]{SS}, we can define a $W_{r,n}$-action on $V^{\otimes n}$: $\mathfrak{S}_n$ acts on $V^{\otimes n}$ by permuting the components of the tensor product, while $s_0$ acts on $V^{\otimes n}$ by letting $
 s_0(v_{\mathbf{i}})=\zeta^{c_1(\mathbf{i})}v_{\mathbf{i}}$
for $\mathbf{i}\in \mathcal{I}(m;n)$.  We denote this representation by $(\Psi_k, V^{\otimes n})$ and let $\chi^n_k$ be its character.
\end{point}

\begin{proposition}[\protect{\cite[Corollary~1.3]{Z24}}]\label{Prop:Chi}For $w\in W_{r,n}$, we have
\begin{eqnarray*}
&&\chi^n_k(w)=(rk+1)^{\ell_n(w)}.
\end{eqnarray*}
 \end{proposition}
We now explain in more details on the relationship between the block characters $\chi_k$ ($0\leq k\leq n$) and the Schur--Weyl duality for $W_{r,n}$.  Let
\begin{equation*}
\mathfrak{g}=\mathfrak{gl}_{k+1}(\mathbb{C})\oplus\mathfrak{gl}_{k}(\mathbb{C})\oplus\cdots
\oplus\mathfrak{gl}_{k}(\mathbb{C})
\end{equation*}
be the Levi subalgebra of the Lie algebra $\mathfrak{gl}_{rk+1}(\mathbb{C})$. Then $\fgl_{k+1}(\mathbb{C})$ and $\fgl_{k}(\mathbb{C})$ acts naturally on $V^{(0)}$ and $V^{(i)}$ ($i=1, \ldots, r-1$), respectively. Further, denote by  $\mathcal{U}(\fgl_{rk+1}(\mathbb{C}))$ the universal enveloping algebra of $\fgl_{rk+1}(\mathbb{C})$. Then the universal enveloping algebra $\mathcal{U}( \mathfrak{g})$ of $\mathfrak{g}$
has the following block decomposition
 \begin{equation}\label{Equ:Block-decom}
 \mathcal{U}(\mathfrak{g})=\mathcal{U}(\fgl_{k+1}
 (\mathbb{C}))\otimes\mathcal{U}(\fgl_{k}(\mathbb{C}))\otimes\cdots\otimes\mathcal{U}
 (\fgl_{k}(\mathbb{C})),
 \end{equation}
which is a subalgebra of $\mathcal{U}(\fgl_{rk+1}(\mathbb{C}))$.
It is known that there is a natural $\mathcal{U}(\fgl_{rk+1}(\mathbb{C}))$-action on $V$, which enables us to yield the natural representation of $\mathcal{U}(\mathfrak{g})$ on $V$ via the restriction. Hence $\mathcal{U}(\mathfrak{g})$ acts naturally on $V^{\otimes n}$ via its Hopf algebra structure. We denote this $\mathcal{U}(\mathfrak{g})$-representation by $(\Gamma_n, V^{\otimes n})$.

Note that irreducible representations of $\mathcal{U}(\mathfrak{gl}_k(\mathbb{C}))$ occurring in $(\mathbb{C}^k)^{\otimes n}$ are parameterized by Young diagrams $\lambda$ of $n$ with $\ell(\lambda)\leq k$. As a consequence, (\ref{Equ:Block-decom}) shows the irreducible representations of $\mathcal{U(\mathfrak{g})}$ occurring in $V^{\otimes n}$ are parameterized by $\boldsymbol{Y}_{r,n}^k$. We denote by $V_{\bgl}$ (resp. $S^{\bgl}$) the irreducible $\mathcal{U(\mathfrak{g})}$-module (resp. $W_{r,n}$-module) corresponding to $\bgl$ and let $\mathbb{C}W_{r,n}$ be the group algebra of $W_{r,n}$ over $\mathbb{C}$.

It is well-known that the following Schur--Weyl duality holds for the $\mathbb{C}W_{r,n}\otimes_{\mathbb{C}}\mathcal{U(\mathfrak{g})}$, see \cite[Proposition~2.8]{SS}.

\begin{proposition}\label{Prop:Schur-Weyl}
Keeping notation as above. Then $\Psi_n(\mathcal{U}(\mathfrak{g}))$
and $\Phi_k(\mathbb{C}W_{r,n})$ are mutually the full centralizer algebras of each other, i.e.,
\begin{equation*}
  \Gamma_n(\mathcal{U}(\mathfrak{g}))=\mathrm{End}_{\mathbb{C}W_{r,n}}(V^{\otimes n}) \text{ and }\Psi_k(\mathbb{C}W_{r,n})=\mathrm{End}_{\mathcal{U}(\mathfrak{g})}(V^{\otimes n}).
\end{equation*}
 Moreover, the $\mathbb{C}W_{r,n}\otimes\mathcal{U}(\mathfrak{g})$-module $V^{\otimes n}$ is decomposed as
 \begin{equation*}
   V^{\otimes n}=\bigoplus_{\bgl\in \boldsymbol{Y}_{r,n}^k}S^{\bgl}\otimes V_{\bgl}.
 \end{equation*}
\end{proposition}
Note that, for $\bgl\in Y_{r,n}^k$, $V_{\bgl}$ has a basis parameterized by $\mathrm{Rstd}_k(\bgl)$.
Thus $\mathrm{dim}_{\mathbb{C}}V_{\bgl}=s_k(\bgl)$ and we yield the following $W_{r,n}$-module isomorphism
\begin{equation}\label{Equ:modules-iso}
V^{\otimes n}\cong \bigoplus_{\bgl\in \boldsymbol{Y}_{r,n}^k} s_k(\bgl)S^{\bgl}=\bigoplus_{\bgl\in \mathscr{P}_{r,n}} s_k(\bgl)S^{\bgl}.
                           \end{equation}

Now we define a signed $W_{r,n}$-action on $V^{\otimes n}$: $\mathfrak{S}_n$ acts on $V^{\otimes n}$ by permuting the components of the tensor product up to a sign, that is , for $a=1, \ldots, n-1$, \begin{equation*}
    \bi s_a:=-(i_1, \ldots, i_{a-1}, i_{a+1}, i_a, i_{a+2}, \ldots, i_n),
\end{equation*}
 and $s_0$ acts on $V^{\otimes n}$ as above. We denote this representation by $(\widetilde{\Psi}_k, V^{\otimes n})$ and let $\widetilde{\chi}^n_k$ be its character.
\begin{proposition}[\protect{\cite[Corollary~1.3]{Z24}}]\label{Prop:Chi-signed}For $w\in W_{r,n}$, we have
\begin{eqnarray*}
&&\widetilde{\chi}^n_k(w)=(-1)^{n+r-1}(-rk-1)^{\ell_n(w)}.
\end{eqnarray*}
 \end{proposition}
 Furthermore the following Schur--Sergeev duality holds for $\mathbb{C}W_{r,n}\otimes_{\mathbb{C}}\mathcal{U(\mathfrak{g})}$, see \cite[Theorem~4.9]{Zdual} and \cite[Remark~2.10]{Z24}.

\begin{proposition}\label{Prop:Schur-Sergeev}
Keeping notation as above. Then $\Psi_n(\mathcal{U}(\mathfrak{g}))$
and $\Phi_k(\mathbb{C}W_{r,n})$ are mutually the full centralizer algebras of each other, i.e.,
\begin{equation*}
  \Gamma_n(\mathcal{U}(\mathfrak{g}))=\mathrm{End}_{\mathbb{C}W_{r,n}}(V^{\otimes n}) \text{ and  }\widetilde{\Psi}_k(\mathbb{C}W_{r,n})
  =\mathrm{End}_{\mathcal{U}(\mathfrak{g})}(V^{\otimes n}).
\end{equation*}
 Moreover, the $\mathbb{C}W_{r,n}\otimes\mathcal{U}(\mathfrak{g})$-module $V^{\otimes n}$ is decomposed as
 \begin{equation*}
   V^{\otimes n}=\bigoplus_{\bgl\in \widetilde{\boldsymbol{Y}}_{r,n}^k}S^{\bgl}\otimes V_{\bgl}.
 \end{equation*}
\end{proposition}
Note that, for $\bgl\in \widetilde{\boldsymbol{Y}}_{r,n}^k$, $V_{\bgl}$  has a basis parameterized by $\mathrm{Cstd}_k(\bgl)$.
Thus $\mathrm{dim}_{\mathbb{C}}V_{\bgl}=c_k(\bgl)$ and we yield the following $W_{r,n}$-module isomorphism
\begin{equation}\label{Equ:modules-iso-sign}
V^{\otimes n}\cong \bigoplus_{\bgl\in \widetilde{\boldsymbol{Y}}_{r,n}^k} c_k(\bgl)S^{\bgl}=\bigoplus_{\bgl\in\mathscr{P}_{r,n}} c_k(\bgl)S^{\bgl}.\end{equation}

Let us remark that Proposition~\ref{Prop:Chi} gives a new interpretation of \cite[Proposition~6]{M17}, which enables us to understand the Foulkes characters of $W_{r,n}$ by applying the Schur--Weyl duality  (Proposition~\ref{Prop:Schur-Weyl}). Similarly, Proposition~\ref{Prop:Chi-signed} gives a new family Foulkes characters of $W_{r,n}$, which can investigated by applying the Schur--Sergeev duality (Proposition~\ref{Prop:Schur-Sergeev}) and provides a duality for the Foulkes characters of $W_{r,n}$.

Now we regard $W_{r,n-1}$ as the subgroup of $W_{r,n}$ generated by $s_0, s_1, \ldots, s_{n-1}$, i.e., we identify $w=w(1)^{c_1}w(2)^{c_2}\cdots w(n\!-\!1)^{c_{n-1}}\in W_{r,n-1}$ with $w(1)^{c_1}w(2)^{c_2}\cdots w(n\!-\!1)^{c_{n-1}}n^0\in W_{r,n}$. For a central function $\chi$ on $W_{r,n}$, the restriction $\chi\negmedspace\downarrow_{r,n-1}$ on $W_{r,n-1}$ is a central function on $W_{r,n-1}$. Further, if $\chi$ is a block function then so is the restriction $\chi\negmedspace\downarrow_{r,n-1}$.
\begin{proposition}[The Branching rule]\label{Prop:Branching-Rule-chi}
For $k=0,1, \ldots, n$, we have
\begin{eqnarray*}
\chi^n_k\negmedspace\downarrow_{r,n-1}=(rk+1)\chi^{n-1}_k &\quad\text {and }\quad&\widetilde{\chi}^n_k\negmedspace\downarrow_{r,n-1}=(rk+1)\widetilde{\chi}^{n-1}_k.
\end{eqnarray*}
\end{proposition}
\begin{proof}Note that for any $w\in W_{r,n-1}$, $\ell_n(w)=\ell_{n-1}(w)+1$. Thanks to Propositions~\ref{Prop:Chi} and \ref{Prop:Chi-signed}, we yield that $\chi^{n}_k(w)=(rk+1)\chi^{n-1}_k(w)$ and $\widetilde{\chi}^n_k(w)=(rk+1)\widetilde{\chi}^{n-1}_k(w)$.
This completes the proof.
\end{proof}
\section{The Foulkes characters}\label{Sec:Foulkes-characters}
In this section we reinterpret the Foulkes characters of $W_{r,n}$ introduced by Miller in \cite{M15} by adopting  Gnedin--Gorin--Kerov's approach in \cite{GGK}, which enables us give a new construction of those $W_{r,n}$-representations affording the Foulkes characters.

Note that the possible values of the function $\ell_n$ on $W_{r,n}$ are integers $0,1, \ldots, n$ and the matrix $\left((ra+1)^{b}\right)_{a,b=0,\ldots, n}$ is invertible. Thus the characters $\chi_0, \ldots, \chi_n$ form a basis of the linear space of block functions on $W_{r,n}$.

Following Miller \cite[Main~Definition and Theorem~5]{M15} and \cite[Definition~2.2]{GGK}, we have the following definition.
\begin{definition}\label{Def:Foulkes}The \emph{Foulkes characters} $\phi^n_0, \phi^n_1, \ldots, \phi^n_{n}$ of $W_{r,n}$ are defined as
\begin{equation}\label{Equ:Foulkes}
  \phi^n_i:=\sum_{j=0}^{i}(-1)^{j}\tbinom{n+1}{j}\chi^n_{i-j}.
\end{equation}
\end{definition}

Clearly $\phi^n_0,\ldots, \phi^n_n$ are block functions on $W_{r,n}$, although it is not obvious whether they are characters of $W_{r,n}$. By applying the same argument as the one of the proof of \cite[Proposition~2.3]{GGK},  we can show the formula (\ref{Equ:Foulkes}) is equivalent to the following formula.
\begin{proposition}For $k=0, 1, \ldots, n$, we have
  \begin{equation}\label{Equ:theta=Foulkes}
  \chi^n_k=\sum_{j=0}^{k}\tbinom{n+j}{j}\phi^n_{k-j}.
\end{equation}
\end{proposition}

Let $\bgl=(\lambda^{(0)};\ldots;\lambda^{(r-1)})$ be an $r$-multipartition of $n$.  Recall that a \emph{standard $\bgl$-tableau} $\mathrm{T}=(\mathrm{T}^{(0)};\ldots;\mathrm{T}^{(r-1)})$ is obtained by putting the integers $1, 2, \ldots, n$ into each box of the Young diagram of $\bgl$ such that the entries of $\mathrm{T}^{(i)}$ increase down columns and along rows for each $i=0,\ldots, r-1$. We denote by  $\mathrm{std}(\bgl)$ the set of standard $\bgl$-tableaux. Given a standard $\bgl$-tableau $\mathrm{T}$, we write $i=(a,b,c)$ and $c_{\mathrm{T}}(i)=c$ when the integer $i$ is lying in the $a$th row and the $b$th column of $\mathrm{T}^{(c)}$.

\begin{definition}\label{Def:Descent-tableau}Let $\bgl\in\mathscr{P}_{r,n}$ and $\mathrm{T}=(\mathrm{T}^{(0)}; \ldots; \mathrm{T}^{(r-1)})$ be a standard $\bgl$-tableau. An integer $i=1, \ldots,n$ is a \textit{descent} of $\mathrm{T}$ when the entries $i=(a,b,c)$, $i+1=(x,y,z)\in\mathrm{T}$ are either $c=z$ and $a<x$, or $c>z$ where $c_{n+1}(\mathrm{T})=0$. The set of descents of $\mathrm{T}$ is denoted $\mathrm{Des}(\mathrm{T})$ and its cardinality is denoted $\mathrm{des}(\mathrm{T})$. Further, we let $\mathrm{des}_i(\mathrm{T})$ be the number of descents of $\mathrm{T}$ that are not less than $i$, that is,
\begin{equation*}
  \mathrm{des}_i(\mathrm{T})=|\mathrm{Des}(w)\cap \{i,\ldots,n\}|,
\end{equation*}
 and let
\begin{equation*}
  f_i(\mathrm{T})=r\mathrm{des}_i(\mathrm{T})+(r-1)-c_i(\mathrm{T}).
\end{equation*}
\end{definition}

Given an $r$-multipartition $\bgl=(\lambda^{(0)};\ldots;\lambda^{(r-1)})$ and fix a standard $\bgl$-tableau $\mathrm{T}$, we say a non-decreasing sequence $\boldsymbol{y}=(y_1,\ldots, y_n)$ are \emph{$\mathrm{T}$-admissible} if $1\leq y_i\leq rk+1$ and $y_i<y_{i+1}$ for every $i\in \mathrm{Des}(\mathrm{T})$. Given a pair $(\mathrm{T},\boldsymbol{y})$ with $\boldsymbol{y}$ being $\mathrm{T}$-admissible, we define a row-semistandard $\boldsymbol{x}_k$-tableau $\mathfrak{t_{\mathrm{T},\boldsymbol{y}}}$ of shape $\bgl$ as follows:
\begin{enumerate}
  \item[(1)] Replace the number $i=(a,b,c)\in \mathrm{T}$ with the number $y_i$ for $1\leq i\leq n$ to obtain a row-semistandard tableau $\widetilde{\mathfrak{t}}_{\mathrm{T},\boldsymbol{y}}$ with entries being numbers;
  \item[(2)]Replace the entries in the $c$-component $\widetilde{\mathfrak{t}}_{\mathrm{T},\boldsymbol{y}}^{(c)}$ with the variables in $\boldsymbol{x}^{(c)}$ up to the ordering for $c=0,\ldots, r-1$ to obtain a row-semistandard $\boldsymbol{x}_k$-tableaux $\mathfrak{t}_{\mathrm{T},\boldsymbol{x}_k}$.
\end{enumerate}
\begin{example}Let $\bgl=((3,2,2);(2,1))$ be a $2$-multipartition of $10$,
\begin{eqnarray*}
&\protect{\mathrm{T}=\left(\twodiagram(1&3&5\cr6&7\cr8&9|2&4\cr10)\right)}
\quad\text{ and }\quad
\protect{\diagram($1$&$\underline{2}$&$3$&$\underline{4}$&$\underline{5}$& 6&$\underline{7}$&8&$9$&$\underline{10}$)}.&
\end{eqnarray*}
Then $\mathrm{Des}(\mathrm{T})=\{2,4,5,7,10\}$, numbers $\{1,2,\ldots,10\}$ with descents of $\mathrm{T}$ marked by underlines, the lexicographically minimal $\mathrm{T}$-admissible sequence $\boldsymbol{y}$ and the corresponding row-semistandard tableau are respectively
\begin{equation*}
 \mathfrak{t}_{\mathrm{T},\boldsymbol{y}}=\left(
 \twodiagram(1&2&3\cr 4&4\cr5&5|1&2\cr 6)\right) \qquad \text{ and }\qquad  \boldsymbol{y}=\diagram(1&1&2&2&3&4&4&5&5&6).
\end{equation*}
The lexicographically minimal $\mathrm{T}$-admissible sequence $\boldsymbol{x}_{\mathrm{T}}$ of variables and the corresponding row-semistandard $\boldsymbol{x}_k$-tableau are respectively
\begin{equation*}
\mathfrak{t}_{\mathrm{T},\boldsymbol{x}_k}=
\left(\protect{\twocdiagram(x_1^{(0)}&x_2^{(0)}&x_3^{(0)}\cr
 x_{4}^{(0)}&x_4^{(0)}\cr x_{5}^{(0)}&x_5^{(0)},x_1^{(1)}&x_2^{(1)}\cr x_3^{(1)})}\right)\quad\text{ and }\quad
\boldsymbol{x}_{\mathrm{T}}=\cdiagram{
x_1^{(0)}&x_1^{(1)}&x_2^{(0)}&x_2^{(1)}&x_3^{(0)}&x_4^{(0)}
&x_4^{(0)}&x_5^{(0)}&x_5^{(0)}&x_3^{(1)}\cr}.
\end{equation*}
 \end{example}
 Furthermore, we have the following fact.

 \begin{lemma}\label{Lemm:Admissible=semistandard}Given $\bgl\in\mathscr{P}_{r,n}$, let $\mathfrak{X}_{\bgl}$ be the set of pairs $(\mathrm{T},\boldsymbol{y})$ where $\mathrm{T}\in \mathrm{std}(\bgl)$ and $\boldsymbol{y}$ is $\mathrm{T}$-admissible. Then the map $(\mathrm{T},\boldsymbol{y})\mapsto \mathfrak{t}_{\mathrm{T},\boldsymbol{y}}$ is bijection between $\mathfrak{X}$ and $\mathrm{Rstd}_k(\bgl)$. In particular, we obtain
 \begin{equation}\label{Equ:s-lambda=m-lambda}
s_{k}(\bgl)=\sum_{j=0}^{k}\tbinom{n+1}{j}m_{k-j}(\bgl),
\end{equation}
where $m_k(\bgl)$ is the number of standard $\bgl$-tableaux with $k$ descents.
 \end{lemma}
\begin{proof}First noticing that for a fixed standard $\bgl$-tableau $\mathrm{T}$ with $k-j$ descents, there are exactly $\tbinom{n+j}{j}$ ways to choose a $\mathrm{T}$-admissible sequence $\boldsymbol{x}_{\mathrm{T}}$.
It is enough to show that the map is a bijection. Indeed, given a pair $(\mathrm{T},\boldsymbol{y})$,  we replace number $i\in \mathrm{T}$ with $y_i$ enables us to obtain uniquely a row-semistandard tableau of shape $\bgl$, then there are exactly one $\mathrm{T}$-admissible sequence $\boldsymbol{x}_{\mathrm{T}}$, which correspondences to a unique row-semistandard $\boldsymbol{x}_k$-tableaux of shape $\bgl$. Conversely, given a row-semistandard $\bx_k$-tableau $\mathfrak{t}$ of shape $\bgl$, its row-reading sequence enables us to obtain uniquely  a minimal non-decreasing sequence $\boldsymbol{y}=(y_1, \ldots,y_n)$ with $1\leq y_i\leq k+1$, which makes us to obtain exactly one standard tableau $\mathrm{T}$ such that $\boldsymbol{y}$ being $\mathrm{T}$-admissible. Thus  we obtain  a pair $(\mathrm{T},\boldsymbol{y})$, where  $\mathrm{T}$ is standard tableau   with $\mathrm{shape}(\mathrm{T})=\mathrm{shape}(\mathfrak{t})=\bgl$ and $\boldsymbol{y}$ is $\mathrm{T}$-admissible. It completes the proof of  the lemma.
\end{proof}
For $\bgl\in\mathscr{P}_{r,n}$, denote by $\chi^{\bgl}$ the character corresponding to the
Specht module $S^{\bgl}$ of $W_{r,n}$ indexed by $\bgl$. Now we can determined explicitly the decomposition of the Foulkes characters of $W_{r,n}$.
 \begin{proposition}\label{Prop:Foulkes-decom}For $k=0,1, \ldots, n$, we have
 \begin{equation*}
   \phi^n_k=\sum_{\bgl\in\mathscr{P}_{r,n}}m_k(\bgl)\chi^{\bgl}.
 \end{equation*}
  \end{proposition}
 \begin{proof}(\ref{Equ:modules-iso}) and Definition~\ref{Def:Foulkes} imply
\begin{equation*}
  \phi^n_k=\sum_{\bgl\in \mathscr{P}_{r,n}}\sum_{j=0}^{k}(-1)^j\tbinom{n+1}{j}s_{k-j}(\bgl)\chi^{\bgl}.
\end{equation*}
Thus it suffices to show that for every $\bgl\in \boldsymbol{Y}_{r,n}$ and $k=0,1, \ldots, n$, we have
\begin{equation*}
 m_k(\bgl)=\sum_{j=0}^{k}(-1)^j\tbinom{n+1}{j}s_{k-j}(\bgl),
\end{equation*}
which is equivalent to the inversion formula~(\ref{Equ:s-lambda=m-lambda}). Thus Lemma~\ref{Lemm:Admissible=semistandard} completes the proof.
 \end{proof}

 Let us remark that the decomposition may be used to study the limit shapes of the Foulkes characters of wreath products by adopting \'{S}niady's argument in \cite{Sn}.

 The following fact follows directly by applying
 Proposition~\ref{Prop:Foulkes-decom},
 which shows that the Foulkes characters are characters of $W_{r,n}$.

 \begin{corollary}For $k=0,1, \ldots, n$,  $\phi^n_k$ is the matrix trace of the
 $W_{r,n}$-representation
 \begin{equation*}
   \Phi^n_k\cong \bigoplus_{\bgl\in\mathscr{P}_{r,n}}m_k(\bgl)S^{\bgl}.
 \end{equation*}
 \end{corollary}

\section{The duality for Foulkes characters}\label{Sec:Duality-Foulkes}
In this section we describe a duality for the Foulkes characters of $W_{r,n}$ via the block characters $\widetilde{\chi}_k$ ($k=0,1,\ldots,n$) and the Schur--Sergeev duality (Proposition~\ref{Prop:Schur-Sergeev}).

Note that the block characters $\widetilde{\chi}_0, \ldots, \widetilde{\chi}_n$ form a basis of the linear space of block functions on $W_{r,n}$. Following Definition~\ref{Def:Foulkes}, we introduce the signed Foulkes characters of $W_{r,n}$.
\begin{definition}\label{Def:Signed-Foulkes}The \emph{signed Foulkes characters} $\widetilde{\phi} _0, \ldots, \widetilde{\phi}_{n}$ of $W_{r,n}$ are defined as
\begin{equation}\label{Equ:Twisted-Foulkes}
  \tilde{\phi}^n_i:=\sum_{j=0}^{i}(-1)^{j}\tbinom{n+1}{j}\widetilde{\chi}^n_{i-j}.
\end{equation}
 \end{definition}
Note that it is not obvious that $\widetilde{\phi} _0, \ldots, \widetilde{\phi}_{n}$ are characters of $W_{r,n}$. Similarly, (\ref{Equ:Twisted-Foulkes}) can be inverted as follows:
 \begin{proposition}For $k=0, 1, \ldots, n$, we have
  \begin{equation}\label{Equ:Signed-chi=Foulkes}
  \widetilde{\chi}^n_{k}=\sum_{j=0}^{k}\tbinom{n+j}{j}\widetilde{\phi}^n_{k-j}.
\end{equation}
\end{proposition}
\begin{definition}Let $\bgl$ be an $r$-multipartition of $n$ and $\mathrm{T}=(\mathrm{T}^{(0)}; \ldots; \mathrm{T}^{(r-1)})$ be a standard $\bgl$-tableau. An integer $i=1, \ldots,n$ is a \textit{column-descent} of $\mathrm{T}$ when the entries $i=(a,b,c)$, $i+1=(x,y,z)\in\mathrm{T}$ are either $c=z$ and $b<y$, or $c<z$. The set of column-descents of $\mathrm{T}$ is denoted $\mathrm{CDes}(\mathrm{T})$ and its cardinality is denoted $\mathrm{cdes}(\mathrm{T})$.
\end{definition}

\begin{remark}\label{Remark:Conjugate-Descent}
Let $\overline{\bgl}$ be the conjugate of the $r$-multipartition $\bgl=(\lambda^{(0)};\lambda^{(1)};\ldots;\lambda^{(r-1)})$, that is, \begin{equation*}
\overline{\bgl}=(\overline{\lambda^{(r-1)}};\ldots;
\overline{\lambda^{(1)}};\overline{\lambda^{(r-1)}}),
\end{equation*}
where $\overline{\lambda^{(i)}}$ is the \emph{conjugate} of  $\lambda^{(i)}$ for $i=0,1,\ldots,r$. It easy to see that if
\begin{equation*}
\mathrm{T}=(\mathrm{T}^{(0)};\mathrm{T}^{(1)};\ldots;\mathrm{T}^{(r-1)})
\end{equation*}
is a standard $\bgl$-tableau, then the conjugate tableau
\begin{equation*}
\overline{\mathrm{T}}=(\overline{\mathrm{T}^{(r-1)}};
\ldots;\overline{\mathrm{T}^{(1)}};\overline{\mathrm{T}^{(0)}})
\end{equation*}
is a standard $\overline{\bgl}$-tableau and $\mathrm{Des}(\mathrm{T})=\mathrm{CDes}(\overline{\mathrm{T}})$.
\end{remark}

Given an $r$-multipartition $\bgl=(\lambda^{(0)};\ldots;\lambda^{(r-1)})$ and fix a standard $\bgl$-tableau $\mathrm{T}$, we say a non-decreasing sequence $\boldsymbol{y}=(y_1,\ldots, y_n)$ are column-$\mathrm{T}$-admissible if $1\leq y_i\leq k+1$ and $y_i<y_{i+1}$ for every $i\in \mathrm{CDes}(\mathrm{T})$. Given a pair $(\mathrm{T},\boldsymbol{y})$ with $\boldsymbol{y}$ being column-$\mathrm{T}$-admissible, we define a column-semistandard $\boldsymbol{x}_k$-tableau $\mathfrak{t}_{\mathrm{T},\boldsymbol{y}}$ of shape $\bgl$ as follows:
\begin{enumerate}
  \item[(1)] Replace the number $i=(a,b,c)\in \mathrm{T}$ with the number $y_i$ for $1\leq i\leq n$ to obtain a column-semistandard tableau $\widetilde{\mathfrak{t}}_{\mathrm{T},\boldsymbol{y}}$ with entries being numbers;
  \item[(2)]Replace the entries in the $c$-component $\mathfrak{t}_{\mathrm{T},\boldsymbol{y}}^{(c)}$ with the variables in $\boldsymbol{x}^{(c)}$ up to the ordering for $c=0,\ldots, r-1$ to obtain an column-semistandard $\boldsymbol{x}_k$-tableaux $\mathfrak{t}_{\mathrm{T},\boldsymbol{x}_k}$.
\end{enumerate}
\begin{example}Let $\bgl=((3,2,2);(2,1))$ be a $2$-multipartition of $10$,
\begin{eqnarray*}
&\protect{\mathrm{T}=\left(\twodiagram(1&3&5\cr 6&7\cr8&9|2&4\cr 10 )\right)}\quad\text{ and }\quad
\protect{\diagram($\underline{1}$&$2$&$\underline{3}$&$4$&$5$& $\underline{6}$&7&$\underline{8}$&$\underline{9}$&$10$)}.&
\end{eqnarray*}
Then $\mathrm{CDes}(\mathrm{T})=\{1,3,6,8,9\}$, numbers $\{1,2,\ldots,10\}$ with descents of $\mathrm{T}$ marked by underlines, the lexicographically minimal $\mathrm{T}$-admissible sequence $\boldsymbol{y}$ and the corresponding row-semistandard tableau are respectively
\begin{equation*}
 \protect{\widetilde{\mathfrak{t}}_{\mathrm{T},\boldsymbol{y}}=
 \left(\twodiagram(1&2&3\cr 3&4\cr4&5|2&3\cr 6)\right)} \qquad \text{ and }\qquad  \boldsymbol{y}=\protect{\diagram(1&2&2&3&3&3&4&4&5&6)}.
\end{equation*}
The lexicographically minimal $\mathrm{T}$-admissible sequence $\boldsymbol{x}_{\mathrm{T}}$ of variables and the corresponding row-semistandard $\boldsymbol{x}_k$-tableau are respectively
\begin{equation*}
\mathfrak{t}_{\mathrm{T},\boldsymbol{x}_k}=\left(
\twocdiagram(x_1^{(0)}&x_2^{(0)}&x_3^{(0)}\cr
 x_{3}^{(0)}&x_4^{(0)}\cr x_{4}^{(0)}&x_5^{(0)},
 x_1^{(1)}&x_2^{(1)}\cr x_3^{(1)})\right)\quad\text{ and }\quad
\boldsymbol{x}_{\mathrm{T}}=\cdiagram{
x_1^{(0)}&x_1^{(1)}&x_2^{(0)}&x_2^{(1)}&x_3^{(0)}&x_3^{(0)}
&x_4^{(0)}&x_4^{(0)}&x_5^{(0)}&x_3^{(1)}\cr}.
\end{equation*}
 \end{example}

The following fact can be proved by applying the same argument as the proof of Lemma~\ref{Lemm:Admissible=semistandard}.
 \begin{lemma}\label{Lemm:Admissible=c-semistandard}Given $\bgl\in\widetilde{\boldsymbol{Y}}^k_{r,n}$, let $\mathfrak{X}_{\bgl}$ be the set of pairs $(\mathrm{T},\boldsymbol{y})$ where $\mathrm{T}\in \mathrm{std}(\bgl)$ and $\boldsymbol{y}$ is $\mathrm{T}$-admissible. Then the map $(\mathrm{T},\boldsymbol{y})\mapsto \mathfrak{c}_{\mathrm{T},\boldsymbol{y}}$ is bijection between $\mathfrak{X}$ and $\mathrm{Cstd}_k(\bgl)$. In particular, we obtain
 \begin{equation}\label{Equ:c-lambda=m-lambda}
c_{k}(\bgl)=\sum_{j=0}^{k}(-1)^j\tbinom{n+1}{j}\overline{m}_{k-j}(\bgl),
\end{equation}
where $\overline{m}_k(\bgl)$ is the number of standard $\bgl$-tableaux with $k$ column-descents.
 \end{lemma}

 Similarly, we have the following decomposition of the signed Foulkes characters.
 \begin{proposition}\label{Prop:Signed-Foulkes-decom}For $k=0,1, \ldots, n$, we have
 \begin{equation*}
   \widetilde{\phi}^n_k=\sum_{\bgl\in\widetilde{\boldsymbol{Y}}^k_{r,n}}
   \overline{m}_k(\bgl)\chi^{\bgl}.
 \end{equation*}
  \end{proposition}
 \begin{proof}(\ref{Equ:modules-iso-sign}) and Definition~\ref{Def:Signed-Foulkes} show
\begin{equation*}
  \widetilde{\phi}^n_k=\sum_{\bgl\in \widetilde{\boldsymbol{Y}}_{r,n}}\sum_{j=0}^{k}(-1)^j\tbinom{n+1}{j}s_{k-j}(\bgl)\chi^{\bgl}.
\end{equation*}
Thus it suffices to show that for every $\bgl\in \widetilde{\boldsymbol{Y}}_{r,n}$
and $k=0,1, \ldots, n$, we have
\begin{equation*}
 \overline{m}_k(\bgl)=\sum_{j=0}^{k}(-1)^j\tbinom{n+1}{j}c_{k-j}(\bgl),
\end{equation*}
which is equivalent to the inversion formula~(\ref{Equ:c-lambda=m-lambda}). Thus Lemma~\ref{Lemm:Admissible=semistandard} completes the proof.
 \end{proof}

 The following fact follows directly by applying
 Proposition~\ref{Prop:Signed-Foulkes-decom},
 which shows that the signed Foulkes characters are characters of $W_{r,n}$.

 \begin{corollary}For $k=0,1, \ldots, n$,  $\widetilde{\phi}_k$ is the matrix trace of the
 $W_{r,n}$-representation
 \begin{equation*}
   \widetilde{\Phi}_k\cong \bigoplus_{\bgl\in\widetilde{\boldsymbol{Y}}^k_{r,n}}\overline{m}_k(\bgl)S^{\bgl}.
 \end{equation*}
 \end{corollary}
The following fact clarifies the relationship between the descents and column-descents.
\begin{lemma}\label{Lemm:Descent-conjugate}If $\bgl\in \boldsymbol{Y}_{r,n}^k$ then
  $m_k(\bar{\bgl})=m_{n-k}(\bgl)$.
\end{lemma}
\begin{proof} Note that if $\mathrm{T}$ is a standard $\bgl$-tableau then every $i=1, \ldots, n$ is either a descent or a column-descent of $\mathrm{T}$. While Remark~\ref{Remark:Conjugate-Descent} shows that a column-descent of $\mathrm{T}$ is a descent  of it conjugate $\overline{\mathrm{T}}$. In particular $\mathrm{des}(T)+\mathrm{des}(\overline{\mathrm{T}})=n$, which implies the lemma.
\end{proof}

 The following fact shows that the signed Foulkes characters are exactly the Foulkes characters, which can be viewed as a duality between the Foulkes characters.

 \begin{theorem}\label{Them:Foulkes=sign-Foulkes}
 For $k=0, 1, \ldots, n$, we have $\phi^n_{k}=\widetilde{\phi}^n_{n-k}$.
 \end{theorem}
 \begin{proof}Note that the map $\bgl\mapsto \bar{\bgl}$ is a bijection between $\boldsymbol{Y}_{r,n}^k$ and $\widetilde{\boldsymbol{Y}}_{r,n}^k$ for $k=0,1, \ldots, n$. Thanks to Proposition~\ref{Prop:Foulkes-decom} and Lemma~\ref{Lemm:Descent-conjugate}, we yield
   \begin{eqnarray*}
   \phi^n_k &=&\sum_{\bgl\in\boldsymbol{Y}^k_{r,n}}m_{k}(\bgl)\chi^{\bgl}\\
   &=&\sum_{\bgl\in\widetilde{\boldsymbol{Y}}^k_{r,n}}m_{n-k}(\bgl)\chi^{\bgl}\\
   &=&\widetilde{\phi}^n_{n-k},
   \end{eqnarray*}
   where the last equality follows by applying Proposition~\ref{Prop:Signed-Foulkes-decom}.
 \end{proof}
 \section{Properties of the Foulkes characters}\label{Sec:Foulkes-Properties}
In  \cite{M15, M17, M21}, Miller showed that the Foulkes characters of $W_{r,n}$ have
some remarkable properties. This section devotes to investigate these properties of the Foulkes characters of $W_{r,k}$ by applying Proposition~\ref{Prop:Foulkes-decom}. In particular, we show that the normalized block characters of $W_{r,n}$ form a simplex whose extreme points are the normalized Foulkes characters of $W_{r,n}$. Furthermore,  we expect the approach will help us to understand the decomposition of the products of the Foulkes characters in terms of the Foulkes characters and  the orthonormal inner product of the Foulkes characters, see  \cite[Theorems~3.14 and 3.11]{M21}.

Note that the irreducible characters $\chi^{\bgl}$, $\lambda\in \mathscr{P}_{r,n}$, form an orthogonal basis in the space of central functions on $W_{r,n}$ endowed with the usual inner product $\langle,\rangle$ and that the block characters $\chi_0, \ldots, \chi_n$ form a basis of the linear space of block functions on $W_{r,n}$. Thanks to Equs.~(\ref{Equ:Foulkes}) and (\ref{Equ:theta=Foulkes}), every block function $\varphi$ on $W_{r,n}$ can be uniquely written as a linear combinations of the Foulkes characters $\phi^n_0,\ldots, \phi^n_n$ of $W_{r,n}$
\begin{equation*}
  \varphi=\sum_{k=0}^n\varphi_k\phi^n_k.
\end{equation*}
The following fact will be useful.
\begin{lemma}\label{Lemm:Lambda-k-descent}
For $k=0,1, \ldots, n$, let $\bgl_k=((n-k);(1^k);\emptyset;\ldots;\emptyset)$. Then $|\mathrm{std}(\bgl_{k})|=\tbinom{n}{k}$ and $\mathrm{des}(\mathrm{T})=k$ for all $\mathrm{T}\in \mathrm{std}(\bgl_{k})$.\end{lemma}
\begin{proof}It is easy to see that $|\mathrm{std}(\bgl_{k})|=\tbinom{n}{k}$. For any $\mathrm{T}\in \mathrm{std}(\bgl_{k})$, there exists $1\leq i_1<i_2<\cdots<i_k\leq n$ such that
\begin{equation*}
  \mathrm{T}=\left(\twocdiagram(y_1&y_2&\cdots&y_{n\!-\!k},
  i_1\cr i_2\cr\vdots\cr i_k)\!\!\!\!\!;\,
  \emptyset;\,\cdots;\emptyset\right),
\end{equation*}
 where $1\leq y_1<y_2<\cdots<y_{n-k}<n$ and $\{y_1,y_2,\cdots,y_{n-k}\}\cup\{i_1,i_2,\cdots,i_k\}=\boldsymbol{n}$. Clearly, $y_1,\ldots, y_{n-k}$ are not descents of $\mathrm{T}$. Assume that $i_k<n$. for $t=1, \ldots, k$, since the number greater than $i_t$ is either lying the first component of $\mathrm{T}$ or lying the below of $i_t$,  $\mathrm{Des}(\mathrm{T})=\{i_1,\i_2, \ldots, i_{k}\}$. Now assume that $i_k=n$. Then the above argument shows that $i_1, \ldots, i_{k-1}$ are descents of $\mathrm{T}$. On the other hand, Definition~\ref{Def:Descent-tableau} shows that $i_k=n$ is a descent of $\mathrm{T}$. As a consequence, $\mathrm{Des}(\mathrm{T})=\{i_1,\i_2, \ldots, i_{k}\}$, which shows the lemma.
\end{proof}

\begin{lemma}\label{Lemm:Linear-expansion} For $k=0, \ldots, n$,
\begin{equation*}
  \varphi_k=\frac{\langle \varphi,\chi^{\bgl_k}\rangle}{\tbinom{n}{k}}.
\end{equation*}
\end{lemma}
\begin{proof}Proposition~\ref{Prop:Foulkes-decom} implies
\begin{eqnarray*}
 \langle\varphi,\chi^{\bgl_k}\rangle&=&\sum_{i=0}^n\sum_{\bgl\in\boldsymbol{Y}_{r,n}^k}
 \varphi_im_i(\bgl)\langle\chi^{\bgl},\chi^{\bgl_k}\rangle\\
 &=&\sum_{i=0}^n\varphi_im_i(\bgl_k)\\
 &=&\tbinom{n}{k}\varphi_k,
 \end{eqnarray*}
 the last equality follows by applying Lemma~\ref{Lemm:Lambda-k-descent}.
\end{proof}

\begin{corollary}\label{Cor:Block-function=character}Let $\varphi$ be a block function on $W_{r,n}$ such that $\varphi=\sum_{k=0}^n\varphi_k\phi^n_k$. Then $\varphi$ is a character of $W_{r,n}$ if and only if $\varphi_k\geq 0$ for $k=0,1, \ldots,n$.
\end{corollary}
\begin{proof}Proposition~\ref{Prop:Foulkes-decom} shows
\begin{equation*}
  \varphi=\sum_{\bgl\in Y_{r,n}}\varphi(\bgl)\chi^{\bgl}
\end{equation*}
where $\varphi(\bgl)\in \mathbb{C}$.
For $k=0,1,\ldots,n$, if $\varphi_k\geq 0$, then coefficients $\varphi(\bgl)$ are non-negative. Since the irreducible characters of $W_{r,n}$ are positive-definite, so is $\varphi$. On the other hand, if $\varphi$ is positive-definite, then the coefficients $\varphi(\bgl)$ are non-negative, in particular, $\varphi(\bgl_k)\geq 0$ for $k=0,1,\ldots,n$, that is, $\varphi_k\geq 0$ for $k=0,1,\ldots,n$.
\end{proof}

Recall that the \emph{row-insertion} takes a tableau $\mathrm{T}$ and a positive integer $a$, and constructs a new tableau, denoted $\mathrm{T}\stackrel{a}{\rightarrow}$. This
tableau will have one more box than $\mathrm{T}$, and its entries will be those of $\mathrm{T}$ together with one more entry labelled $a$, but there is some moving around.
The recipe is as follows:
if $a$ is at least as large as all the entries in the first
row of $\mathrm{T}$, simply add $a$ in a new box to the end of the first row. If not,
find the left-most entry in the first row that is strictly larger than $a$. Put $a$ in
the box of this entry, and remove (``bump") the entry. Take this entry that was
bumped from the first row, and repeat the process on the second row. Keep
going until the bumped entry can be put at the end of the row it is bumped into,
or until it is bumped out the bottom, in which case it forms a new row with
one entry.

Recall that the Robinson--Schensted--Knuth (RSK) correspondence for $W_{r,n}$  is a bijection between $W_{r,n}$ and pairs of standard Young tableaux $(\mathrm{S},\mathrm{T})$ with $\mathrm{shape}(\mathrm{S})=\mathrm{shape}(\mathrm{T})=\bgl$ for all $r$-multipartition $\bgl$ of $n$, where $\mathrm{T}$ is called the \textit{recording tableau}). More precisely, for any $w=i_1^{c_1}i_2^{c_2}\cdots i_n^{c_n}\in W_{r,n}$, the RSK correspondence constructs the pairs of standard tableaux $(\mathrm{S},\mathrm{T})=(\mathrm{S}(w),\mathrm{T}(w))$ iteratively:
 \begin{equation*}
   (\emptyset,\emptyset)=(\mathrm{S}_0,\mathrm{T}_0), (\mathrm{S}_1,\mathrm{T}_1), \ldots, (\mathrm{S}_n,\mathrm{T}_n)=(\mathrm{S},\mathrm{T})
 \end{equation*}
 by applying the following rules:
\begin{enumerate}
  \item[(1)] $\mathrm{S}_j$ and $\mathrm{T}_j$ are standard tableaux containing $j$ boxes with $\mathrm{shape}(\mathrm{S}_j)=\mathrm{shape}(\mathrm{T}_j)$;

  \item[(2)] $\mathrm{S}_j$ is obtained from $\mathrm{S}_{j-1}$ by row-inserting $i_j^{c_j}$ into the $c_j$-component of $\mathrm{S}_{j-1}$ of $\mathrm{S}_{j-1}$;

\item[(3)] $\mathrm{T}_j$ is obtained  from $\mathrm{T}_{j-1}$ by inserting number $j$ in the newly added box.
           \end{enumerate}
To show that this RSK correspondence is a bijection, we can construct the inverse algorithm by using the RSK correspondence for partitions in the reverse order of the entries of $\mathrm{T}$, and using the component of $S$ to determine the uninserted number. We denote this bijection by
 \begin{equation*}
  w\xymatrix@C=0.65cm{\ar@{<->}[r]^{\substack{\mathrm{RSK}}}&}
   (\mathrm{S},\mathrm{T}).
\end{equation*}

\begin{example}\label{Exam:RSK-W}Let $w=3^0\,2^0\,1^0\,4^2\,6^2\,5^1\in W_{3,6}$. Then the RSK correspondence is as following:
\begin{eqnarray*}
&&(\emptyset)\!\!\stackrel{3^0}{\rightarrow}\!\!
\protect{\left(\diagram(3);\emptyset;\emptyset\right)}
\!\!\stackrel{2^0}{\rightarrow}\!\!
\protect{\left(\diagram(2\cr3);\emptyset;\emptyset\right)}
\!\!\stackrel{1^0}{\rightarrow}\!\!
\protect{\left(\diagram(1\cr2\cr3);\emptyset;\emptyset\right)}
\!\!\stackrel{4^2}{\rightarrow}\!\!
\protect{\left(\diagram(1\cr2\cr3);\emptyset;\diagram(4)\right)}
\!\!\stackrel{6^2}{\rightarrow}\!\!
\protect{\left(\diagram(1\cr2\cr3);\emptyset;\diagram(4&6)\right)}
\!\!\stackrel{5^1}{\rightarrow}\!\!
\protect{\left(\diagram(1\cr2\cr3);\diagram(5);\diagram(4&6)\right)};\\
&&(\emptyset)\!\!\stackrel{1}{\rightarrow}\!\!
\protect{\left(\diagram(1);\emptyset;\emptyset\right)}
\!\!\stackrel{2}{\rightarrow}\!\!
\protect{\left(\diagram(1\cr2);\emptyset;\emptyset\right)}
\!\!\stackrel{3}{\rightarrow}\!\!
\protect{\left(\diagram(1\cr2\cr3);\emptyset;\emptyset\right)}
\!\!\stackrel{4}{\rightarrow}\!\!
\protect{\left(\diagram(1\cr2\cr3);\emptyset;\diagram(4)\right)}
\!\!\stackrel{5}{\rightarrow}\!\!
\protect{\left(\diagram(1\cr2\cr3);\emptyset;\diagram(4&5)\right)}
\!\!\stackrel{6}{\rightarrow}\!\!
\protect{\left(\diagram(1\cr2\cr3);\diagram(6);\diagram(4&5)\right)}.
\end{eqnarray*}
Thus \begin{equation*}
3^0\,2^0\,1^0\,4^2\,6^2\,5^1\xymatrix@C=0.65cm{\ar@{<->}[r]^{\substack{\mathrm{RSK}}}&} =
\left(\left(\diagram(1\cr2\cr3);\diagram(5);\diagram(4&6)\right), \left(\diagram(1\cr2\cr3);\diagram(6);\diagram(4&5)\right)\right).
\end{equation*}
\end{example}
Noticing that in the above example, we have $\mathrm{Des}(\mathrm{T})=\{1,2, 5,6\}=\mathrm{Des}(w)$. Indeed, we have the following fact.
\begin{lemma}\label{Lemm:RSK-descent}For any $w\in W_{r,n}$, if $w\xymatrix@C=0.65cm{\ar@{<->}[r]^{\substack{\mathrm{RSK}}}&}
   (\mathrm{S},\mathrm{T})$, then $\mathrm{Des}(w)=\mathrm{Des}(\mathrm{T})$.
\end{lemma}
\begin{proof}Assume that $w=w(1)^{c_1}w(2)^{c_2}\cdots w(n)^{c_n}\in W_{r,n}$.  For $i\in \mathrm{Des}(w)$, since $w(n+1)=n+1$ and $c_{n+1}=0$, we have the following two cases:
\begin{enumerate}
  \item[(a)] $w(i)>w(i+1)$ and $c_{i}=c_{i+1}$. The lemma follows by applying Row Bumping Lemma in \cite[\S1.1]{Fulton}.
  \item[(b)]$c_i>c_{i+1}$. The lemma follows directly the RSK correspondence and Definition~\ref{Def:Descent-tableau}
\end{enumerate}
 This completes the proof.
\end{proof}

Now we can reprove the following properties of the Foulkes characters of $W_{r,n}$, which were first proved by Miller in \cite{M15}.
\begin{theorem}\label{Them:Properties}Let $\phi_0,\phi_1, \ldots,\phi_n$ be the Foulkes characters of $W_{r,n}$. Then
 \begin{enumerate}
   \item[(a)]$\phi_k(\boldsymbol{e})=E_{r,n}(k)$ for $k=0,1,\ldots,n$.
   \item[(b)]$\phi_0+\phi_1+\cdots+\phi_n=\chi_{\mathrm{reg}}$, where $\chi_{\mathrm{reg}}$ is the regular character of $W_{r,n}$.
\item[(c)]$\phi^n_k\!\downarrow_{r,n-1}=\left((n+1)r-(rk+1)\right)\phi^{n-1}_{k-1}+(rk+1)\phi^{n-1}_k$.
 \end{enumerate}
 \end{theorem}
\begin{proof}(a)  Lemma~\ref{Lemm:RSK-descent} shows the RSK sends elements of $W_{r,n}$ with $k$ descents to pairs $(\mathrm{S},\mathrm{T})$ of standard Young tableaux with same shape and $\mathrm{des}(\mathrm{T})=k$. Thus
 \begin{equation*}
  E_{r,n}(k)=\sum_{\bgl\in\rpn}m_k(\bgl)|\mathrm{std}(\bgl)|.
\end{equation*}
For $k=0,1, \ldots, n$,  Proposition~\ref{Prop:Foulkes-decom} shows
\begin{eqnarray*}
  \phi_k(\boldsymbol{e})&=&\sum_{\bgl\in\boldsymbol{Y}^k_{r,n}}m_k(\bgl)\chi^{\bgl}(\boldsymbol{e})\\
   &=&\sum_{\bgl\in\boldsymbol{Y}^k_{r,n}}m_k(\bgl)|\mathrm{std}(\bgl)|\\
   &=&E_{r,n}(k),
\end{eqnarray*}
where  the second equality follows by noticing that  the dimension $\chi^{\bgl}(1)$ of the irreducible representation $S^{\bgl}$ equals to the number of standard $\bgl$-tableaux.

 (b) Given an $r$-multipartition $\bgl$ of $n$, it is easy to see that
 \begin{equation*}
   \mathrm{std}(\bgl)=\sqcup_{k=0}^n\{\mathrm{T}\in \mathrm{std}(\bgl)|\mathrm{des}(\mathrm{T})=k\},
 \end{equation*}
 that is, $|\mathrm{std}(\bgl)|=\sum_{k=0}^nm_k(\bgl)$.
 Thanks to Proposition~\ref{Prop:Foulkes-decom}, we have
\begin{eqnarray*}
\sum_{k=0}^n\phi_k &=& \sum_{k=0}^n\sum_{\bgl\in\boldsymbol{Y}_{r,n}^k}m_{k}(\bgl)\chi^{\bgl} \\
&=&\sum_{\bgl\in\boldsymbol{Y}_{r,n}} \sum_{k=0}^nm_{k}(\bgl)\chi^{\bgl} \\
 &=&\sum_{\bgl\in\rpn}|\mathrm{std}(\bgl)|\chi^{\bgl}\\
 &=&\chi_{\mathrm{reg}},
\end{eqnarray*}
where the last equality follows by applying the Artin--Wedderburn theorem.

(c).  For $k=0,1,\ldots, n$ and $i=0,\ldots, k$, the proof of \cite[Proposition~6.2]{GGK} shows
\begin{equation*}
 \tbinom{n+1}{k-i}=k\tbinom{n}{k-i}-(n-k+1)\tbinom{n+1}{k-i-1},
\end{equation*}
which implies
\begin{eqnarray*}
(ri+1)\tbinom{n+1}{k-i}&=&ri\tbinom{n+1}{k-i}+\tbinom{n+1}{k-i}\\
&=&r\left(k\tbinom{n}{k-i}-(n-k+1)\tbinom{n+1}{k-i-1}\right)
+\tbinom{n}{k-i}+\tbinom{n}{k-i-1}\\
&=&(rk+1)\tbinom{n}{k-i}-\left(r(n+1)-(rk+1)\right)\tbinom{n}{k-i-1}.
\end{eqnarray*}
Now combining Equ.~(\ref{Equ:Foulkes}) and Proposition~\ref{Prop:Branching-Rule-chi}, we obtain that
\begin{eqnarray*}
\phi_k^n\!\downarrow_{r,n-1}&=&
\sum_{i=0}^k(-1)^{k-i}\tbinom{n+1}{k-i}\chi^n_{i}\!\downarrow_{r,n-1}\\
&=&\sum_{i=0}^k(-1)^{k-i}(ri+1)\tbinom{n+1}{k-i}\chi^{n-1}_{i}\\
&=&\sum_{i=0}^k(-1)^{k-i}\left((rk+1)\tbinom{n}{k-i}-\left(r(n+1)-
(rk+1)\right)\tbinom{n}{k-i-1}\right)\chi^{n-1}_{i}\\
&=&(rk+1)\sum_{i=0}^k(-1)^{k-i}\tbinom{n}{k-i}\chi^{n-1}_{i}+\left(r(n+1)-
(rk+1)\right)\sum_{i=0}^{k-1}(-1)^{k-i-1}\tbinom{n}{k-i-1}\chi^{n-1}_{i}\\
&=&(rk+1)\phi_k^{n-1}+\left(r(n+1)-
(rk+1)\right)\phi_{k-1}^{n-1}.
\end{eqnarray*}
The theorem is proved.
\end{proof}
\begin{remark}(a) can be proved by applying (c). Indeed, by evaluating at unity, we obtain the relation
  \begin{equation*}
    \phi_k^n(\boldsymbol{e})=(rk+1)\phi^{n-1}_k(\boldsymbol{e})+
    \left(r(n+1)-(rk+1)\right)\phi^{n-1}_k(\boldsymbol{e}),
  \end{equation*}
  which is the recurrence Equ.~(\ref{Equ:Euler-recurrence}) for the Eulerian number.
\end{remark}

Let $\bgl$ be an $r$-multipartition of $n$. The branching rule \cite[Theorem 4.1]{Okada} says
\begin{equation}\label{Equ:Reps-Branching-law}
 \chi^{\bgl}\!\downarrow_{r,n-1}=\sum_{\Box\in\mathscr{R}(\bgl)}\chi^{\bgl-\Box},
\end{equation}
where $\bgl-\Box$ means that the $r$-multipartition of $n-1$ is obtained from $\bgl$ by removing the removed box $\Box$.
\begin{remark}Given an $r$-multipartition $\bgl$ of $n$ and $\Box\in\mathscr{R}(\bgl)$,
   Proposition~\ref{Prop:Foulkes-decom}, Theorem~\ref{Them:Properties}(c) and Equ.~(\ref{Equ:Reps-Branching-law}) show,  for $k=0,1,\ldots, n$,
\begin{eqnarray*}
&&m_k(\bgl)=\left(r(n+1)
-(rk+1)\right)m_{k-1}(\bgl\!-\!\Box)+(rk+1)m_{k}(\bgl\!-\!\Box).
\end{eqnarray*}
It would be interesting to give a combinatorial proof of this identity.
\end{remark}

Clearly the set of all block characters is a convex cone of non-negative linear combinations of the extreme normalized block characters. Furthermore, we have the following fact.
\begin{corollary}Let $\mathscr{N}_{r,n}$  be set of normalized block characters of $W_{r,n}$. Then $\mathscr{N}_{r,n}$ is a simplex and its extreme points are the normalized Foulkes characters $\frac{\phi^n_0}{\phi^n_0(\boldsymbol{e})},\ldots,\frac{\phi^n_n}
{\phi^n_n(\boldsymbol{e})}$.
\end{corollary}
\begin{proof}Clearly $\mathscr{N}_{r,n}$ is a simplex. Now we show that the normalized Foulkes characters are extreme points of $\mathscr{N}_{r,n}$. For $k=0,1,\ldots,n$, assume that there exist $a\in (0,1)$, $\rho,\tau\in \mathscr{N}_{r,n}$ such that
\begin{equation*}
  \frac{\phi^n_k}{\phi^n_k(\boldsymbol{e})}=(1-a)\rho+a\tau,
\end{equation*}
then Lemma~\ref{Lemm:Linear-expansion} shows $\langle(1-a)\rho+a\tau,\chi^{\bgl_j}\rangle=0$ for $j\neq k$, which implies that
\begin{equation*}
\langle\rho,\chi^{\bgl_j}\rangle=0=\langle\tau,\chi^{\bgl_j}\rangle\text{ for }j\neq k.
\end{equation*}
Thus $\rho=\tau=\frac{\phi^n_k}{\phi^n_k(\boldsymbol{e})}$, i.e., $\frac{\phi^n_0}{\phi^n_0(\boldsymbol{e})},\ldots,\frac{\phi^n_n}{\phi^n_n(\boldsymbol{e})}$ are extreme points of $\mathscr{N}_{r,n}$.
\end{proof}
The following fact enables us to obtain  a block character on $W_{r,n}$.
\begin{lemma}\label{Lemm:q-block}Let $q\geq 0$ be a parameter. Then
\begin{eqnarray*}
  &&(rq+1)^{\ell_n(w)}=\sum_{j=0}^n\tbinom{q+n-j}{n}\phi^n_j(w) \text { for $w\in W_{r,n}$},
\end{eqnarray*}
 where $\tbinom{q+n-j}{n}$ is the generalized binomial coefficient involves parameter $q$.
\end{lemma}

\begin{proof}For $w\in W_{r,n}$, let
\begin{equation*}
f(q)=(rq+1)^{\ell_n(w)}-\sum_{j=0}^n\tbinom{q+n-j}{n}\phi^n_j(w).
\end{equation*}
 Then $f(q)$ is a polynomial in variable $q$  with degree no more than $n$, then the equality follows by noticing that $f(k)=0$ for $k=0, \ldots, n$.
\end{proof}
For $q\geq 0$, the function
\begin{equation*}
E: W_{r,n}\rightarrow \mathbb{C}(q),\quad w\mapsto \frac{(rq+1)^{\ell_n(w)}}{(rq+1)(r(q+1)+1)\cdots(r(q+n)+1)}
\end{equation*}
is a probability on $W_{r,n}$, which is the wreath analogue of the Ewens' distribution on the symmetric group.

\begin{corollary}Assume that $q$ is not an integer. Then $(rq+1)^{\ell_n(w)}$ is a character of $W_{r,n}$ if and only if $q>n$.
\end{corollary}

\begin{proof}Lemma~\ref{Lemm:q-block} shows $(rq+1)^{\ell(w)}$ is a block function on $W_{r,n}$. Thanks to Corollary~\ref{Cor:Block-function=character}, it is a block character if and only if $\tbinom{q+n-j}{n}\geq 0$ for $j=0, 1, \ldots, n$, i.e., $q>n$.
\end{proof}

  \section{The coinvariant algebra}\label{Sec:Coinvariant}
Using the  the descent monomial basis for this coinvariant algebras of wreath
products introduced by Bango--Biagioli in \cite{BB}, this section devotes to present an explicit construction of the representations $\Phi^n_k$ of $W_{r,n}$ with traces $\phi^n_k$ via the coinvariant algebra of $W_{r,n}$.

\begin{point}{}*
Let $\mathbb{C}[\boldsymbol{x}]$ be the algebra of polynomials in variables $\boldsymbol{x}=x_1, \ldots, x_n$. The wreath product $W_{r,n}$ acts on $\mathbb{C}[\boldsymbol{x}]$ as follows: for $w=w(1)^{c_1}\cdots w(n)^{c_n}\in W_{r,n}$,
\begin{equation*}
  wf(x_1, \ldots, x_n):=f(\zeta^{c_1}x_{w(1)}, \ldots, \zeta^{c_n}x_{w(n)}).
\end{equation*}
Alternatively, when we view $W_{r,n}$ as the set of $n\times n$ matrices with exactly one non-zero entry in each row and column where the non-zero entries are $r$th roots of unity. The action of $W_{r,n}$ on $\mathbb{C}[\boldsymbol{x}]$ in this case is matrix multiplication.

Let $\mathcal{I}_n$ be the ideal of $\mathbb{C}[\boldsymbol{x}]$ generated by constant-term-free  $W_{r,n}$-invariant polynomials. Any $W_{r,n}$-invariant polynomials must be a symmetric polynomial in variables $x_1^r, \ldots, x_n^r$. Denote this set of variables as $\boldsymbol{x}^r$. Then $\mathcal{I}_n=\langle e_1(\boldsymbol{x}^r), \ldots,  e_n(\boldsymbol{x}^r)\rangle$, where
\begin{equation*}
 e_d(\boldsymbol{x}^r)=e_d(x_1^r, \ldots, x_n^r):=\sum_{1\leq i_1\leq\cdots\leq i_d} x_{i_1}^r\cdots x_{i_d}^r
\end{equation*}
 is the $d$th elementary symmetric functions in variable powers $x_1^r, \ldots, x_n^r$. Further, let $\mathbb{C}[\boldsymbol{x}]^{\mathrm{coinv}}$ be the coinvariant algebra associated to $W_{r,n}$, that is, $\mathbb{C}[\boldsymbol{x}]^{\mathrm{coinv}}$ is the quotient algebra
 \begin{equation*}
  \mathbb{C}[\boldsymbol{x}]^{\mathrm{coinv}}=\mathbb{C}[\boldsymbol{x}]/\mathcal{I}_n.
 \end{equation*}
 Denote by $\pi$ the canonical projection $\pi: \mathbb{C}[\boldsymbol{x}]\rightarrow \mathbb{C}[\boldsymbol{x}]^{\mathrm{coinv}}$. Observe that $\mathbb{C}[\boldsymbol{x}]^{\mathrm{coinv}}$ inherits from $\mathbb{C}[\boldsymbol{x}]$ the structure of a $W_{r,n}$-module. Chevalley's result \cite{Chevalley} implies that $\mathbb{C}[\boldsymbol{x}]^{\mathrm{coinv}}\cong \mathbb{C}W_{r,n}$ as ungraded $W_{r,n}$-modules.
\end{point}

Following Bagno and Biagioli \cite{BB}, for $w=w(1)^{c_1}w(2)^{c_2}\cdots w(n)^{c_n}\in W_{r,n}$, we let $\mathrm{des}_i(w)$ be the number of descents of $w$ that are not less than $i$, that is,
\begin{equation*}
  \mathrm{des}_i(w)=|\mathrm{Des}(w)\cap\{i, i+1, \ldots,n\}|,
\end{equation*}
and define the flag descent values as
\begin{equation*}
  f_i(w)=r\mathrm{des}_i(w)+(r-1)-c_i.
\end{equation*}

\begin{example}Let $w=3^0\,2^0\,1^0\,4^2\,6^2\,5^1\in W_{3,6}$. Then $\mathrm{Des}(w)=\{1,2,5,6\}$,
\begin{eqnarray*}
&&(\mathrm{des}_1(w), \mathrm{des}_2(w), \mathrm{des}_3(w),\mathrm{des}_4(w),\mathrm{des}_5(w),\mathrm{des}_6(w))=(4,3,2,2,2,1),\\
&&(f_1(w), f_2(w),f_3(w),f_4(w),f_5(w),f_6(w))=(14,11,8,6,6,4).
\end{eqnarray*}
\end{example}

Now we recall the $r$-descent monomials in $\mathbb{C}[\boldsymbol{x}]$ introduced by Bango and Biagioli in \cite{BB}.

\begin{definition}Given $w=w(1)^{c_1}w(2)^{c_2}\cdots w(n)^{c_n}\in W_{r,n}$, we define the \textit{$r$-descent monomials} $m_w$ as follows:
\begin{equation*}
  m_w=x_{w(1)}^{f_1(w)}x_{w(2)}^{f_2(w)}\cdots x_{w(n)}^{f_n(w)}.
\end{equation*}
\end{definition}
Note that by the definitions of $f_i$ and of descents of $r$-colored permutations that
\begin{equation*}
f(w):=(f_1(w), \ldots,f_n(w))
\end{equation*}
is a weakly decreasing sequence such that $f_i(w)-f_{i+1}(w)\leq r$ for $w\in W_{r,n}$ and $1\leq i<n$. Thus $\mathrm{deg}(m_w)=f(w)$.

Thanks to \cite[Theorem~4.3]{BB} and Remark~\ref{Remark:Descent}, we can prove the following fact, which is also can be proved by applying arguments similar to those given in the proof of \cite[Lemma~3.3]{B-Caselli}.

\begin{theorem}\label{Them:Descent-basis}
Keeping notations as above. Then $\{m_w+\mathcal{I}_n|w\in W_{r,n}\}$ is a $\mathbb{C}$-linear basis of $\mathcal{R}_n^r$. Furthermore, if $f$ is a polynomial belonging to $m_w+\mathcal{I}_n$, then $\mathrm{deg}(m_w)\preceq\mathrm{deg}(f)$.
\end{theorem}

Given an sequence $\boldsymbol{d}=(d_1, \ldots, d_n)$ of non-negative integers, let $\boldsymbol{d}_{+}$ be the partition obtained by rearranging the integers $d_1, \ldots, d_n$ in a weakly decreasing order. In what follows we write $\lambda\prec\mu$ for two partitions $\lambda$ and $\mu$ if $\lambda=(\lambda_1,\lambda_2,\ldots)$ precedes in the lexicographical order. For a polynomial
\begin{equation*}
  f=\sum_{\boldsymbol{d}}a_{\boldsymbol{d}}x_1^{d_1}\cdots x_n^{d_n}\in \mathbb{C}[\boldsymbol{x}]^{\mathrm{coinv}},
\end{equation*}
its degree $\mathrm{deg}(f)$ is defined as the minimal partition $\mu$ such that $b_{\boldsymbol{d}}=0$ when $\mu\prec\boldsymbol{d}_{+}$.

For $k=0,1, \ldots$, we let
\begin{eqnarray*}
  &&\mathcal{F}_n(k):=\{f\in \mathbb{C}[\boldsymbol{x}]|\mathrm{deg}_1(f)\leq k\},
\end{eqnarray*}
where $\mathrm{deg}_1(f)$ is the largest part of $\mathrm{deg}(f)$. Clearly $\mathcal{F}_n(k)$ is a $W_{r,n}$-module and denote by \begin{equation*}
  \mathcal{F}_n^r(k)=\pi(\mathcal{F}_n(k)).
\end{equation*}
Then $\mathcal{F}_n^r(k)$ is a $W_{r,n}$-submodule of $\mathbb{C}[\boldsymbol{x}]^{\mathrm{coinv}}$ and Theorem~\ref{Them:Descent-basis} shows
 \begin{equation}\label{Equ:F-n-r-k}
  \mathcal{F}_n^r(k)=\mathrm{span}_{\mathbb{C}}\{m_w|f_1(w)\leq k\}.
\end{equation}

Let $\boldsymbol{q}=(q_1,\ldots,q_n)$ be a sequence of formal variables.
For an element $w\in W_{r,n}$ let the (graded) trace of its action on  $\mathbb{C}[\boldsymbol{x}]$ be
\begin{equation*}
\mathrm{Tr}_{\mathbb{C}[\boldsymbol{x}]}(w):=\sum_{m}\langle wm,m\rangle \boldsymbol{q}^{\lambda(m)},
\end{equation*}
where the sum is over all monomials $m\in \mathbb{C}[\boldsymbol{x}]$, $\lambda(m)$ is the exponent partition of $m$, and the inner product is such that the set of all monomials is an orthonormal basis for $\mathbb{C}[\boldsymbol{x}]$. Note that
$\langle wm,m\rangle\in\{0, \varsigma^i|i=1, \ldots, r-1\}$.

Similarly, for an element $w\in W_{r,n}$ let the (graded) trace of its action on  $\mathbb{C}[\boldsymbol{x}]^{\mathrm{coinv}}$ be
\begin{equation*}
\mathrm{Tr}_{\mathbb{C}[\boldsymbol{x}]^{\mathrm{coinv}}}(w):=\sum_{g\in W_{r,n}}\langle w(m_g+\mathcal{I}_n),m_g+ I_n\rangle \boldsymbol{q}^{f_g},
\end{equation*}
where the inner product is such that the set of all $r$-descent monomials  is an orthonormal basis for $\mathbb{C}[\boldsymbol{x}]^{\mathrm{coinv}}$.

For $w\in W_{r,n}$, thanks to \cite[LEMMA~10.4]{BB} and \cite[Proposition~4.14]{Meyer}, we have
\begin{equation}\label{Equ:Graded-trace}
 \prod_{i=1}^n(1-q_1^r\cdots q_i^r)\mathrm{Tr}_{\mathbb{C}[\boldsymbol{x}]}(w)=
  \mathrm{Tr}_{\mathbb{C}[\boldsymbol{x}]^{\mathrm{coinv}}}(w)
 =\sum_{\bgl\in Y_{r,n}}\chi^{\bgl}(w)\sum_{\mathrm{T}\in \mathrm{std}(\bgl)}\prod_{i=1}^nq_i^{f_i(\mathrm{T})}.
\end{equation}

The following theorem clarify the relevance of the filtration $\mathcal{F}^r_n(k)$ to the block characters, which gives another way to prove that the dimension of $\phi^n_k$ is $E_{r,n}(k)$.

\begin{theorem}\label{Them:Filtration=Foulkes}
For $k=0, 1, \ldots,n$, there is an isomorphism of $W_{r,n}$-modules
\begin{equation*}
   \mathcal{F}^r_n(k)\cong \bigoplus_{i=0}^k\Phi^n_i
\end{equation*}
 and $\Phi^n_k\cong  \mathcal{F}^r_n(k)/\mathcal{F}^r_n(k-1)$ with $\mathcal{F}^r_n(-1)=0$. In particular, for $k=0,1,\ldots,n$,
 \begin{equation*}
   \mathrm{Trace}(\mathcal{F}^r_n(k)/\mathcal{F}^r_n(k-1))=\phi_k^n.
 \end{equation*}
\end{theorem}
\begin{proof}
Letting $q_1=q$ and $q_2=\cdots=q_n=1$, Equ.~(\ref{Equ:Graded-trace}) shows
\begin{eqnarray*}
\sum_{g\in W_{r,n}}\langle w(m_g),m_g\rangle q^{f_1(g)}&=&
\sum_{k=0}^nq^{k}\sum_{\bgl\in Y_{r,n}}\chi^{\bgl}(w)m_k(\bgl).
\end{eqnarray*}
Now let $\mathrm{Trace}_k^n$ be the trace of the $W_{r,n}$-representation $\mathcal{F}^{r}_n(k)/\mathcal{F}^{r}_n(k\!-\!1)$. Then Equ.~(\ref{Equ:F-n-r-k}) shows
\begin{eqnarray*}
\sum_{g\in W_{r,n}}\langle w(m_g),m_g\rangle q^{f_1(g)}&=&\sum_{k=0}^nq^k\mathrm{Trace}_k^n(w).
\end{eqnarray*}
Therefore, for any $w\in W_{r,n}$, we have
\begin{eqnarray*}
\sum_{k=0}^nq^{k}\sum_{\bgl\in Y_{r,n}}\chi^{\bgl}(w)m_k(\bgl)&=&\sum_{k=0}^nq^k\mathrm{Trace}_k^n(w),
\end{eqnarray*}
which implies
\begin{eqnarray*}
\mathrm{Trace}_k^n(w)&=&\sum_{\bgl\in Y_{r,n}}\chi^{\bgl}(w)m_k(\bgl) \text{ for any }w\in W_{r,n}.
\end{eqnarray*}
Thus the theorem follows by applying Proposition~\ref{Prop:Foulkes-decom}.
\end{proof}

\end{document}